\documentclass[reqno]{amsart}

\usepackage{graphicx}
  
\usepackage[left=4cm, right=4cm, top=2cm]{geometry}
\usepackage[table,dvipsnames]{xcolor}
\usepackage{array}
\usepackage{subcaption}
\usepackage{todonotes}
\usepackage{bbm}
\usepackage{hyperref}
\usepackage{amsmath, amsthm, amssymb}
\usepackage{cancel}
\usepackage{multirow}
\usepackage{pgfplots}
\usepackage{tikz}
\usepackage{tikzit}
\usetikzlibrary{tikzmark}

\tikzstyle{new style 0}=[fill={rgb,255: red,255; green,16; blue,20}, draw=black, shape=circle]
\tikzstyle{new style 1}=[fill={rgb,255: red,66; green,255; blue,33}, draw=black, shape=circle]
\tikzstyle{new style 2}=[fill={rgb,255: red,196; green,58; blue,255}, draw=black, shape=circle]
\tikzstyle{new style 3}=[fill={rgb,255: red,10; green,104; blue,255}, draw=black, shape=circle]
\tikzstyle{new style 4}=[fill=white, draw={rgb,255: red,59; green,141; blue,21}, shape=circle]

\tikzstyle{new edge style 0}=[<->, fill=white]
\tikzstyle{new edge style 1}=[<-]
\tikzstyle{new edge style 2}=[-, draw={rgb,255: red,48; green,30; blue,255}, fill=none]
\tikzstyle{new edge style 3}=[-, draw={rgb,255: red,67; green,255; blue,15}]
\tikzstyle{purple}=[-, draw={rgb,255: red,138; green,20; blue,255}]
\tikzstyle{new edge style 4}=[-, draw={rgb,255: red,250; green,0; blue,0}]
\tikzstyle{new edge style 5}=[dashed, fill=none, draw=black, -]

\usepackage{float}
\usepackage{microtype}
\usepackage{adjustbox}
\usepackage{tcolorbox}
\usepackage{xifthen}

\pgfplotsset{compat=newest}
\usepgfplotslibrary{groupplots}
\usepgfplotslibrary{dateplot}
\usetikzlibrary{math}
\usetikzlibrary{shapes,arrows}
\usetikzlibrary{matrix}
\tikzset{
  block/.style    = {draw, thick, rectangle, minimum height = 3em, minimum width = 3em},
  causalvar/.style      = {draw, circle, node distance = 2cm}
}

\newtheorem{ass}{Assumption}

\newtheorem{defi}{Definition}
\newtheorem*{defi*}{Definition}
\newtheorem{remark}{Remark}
\newtheorem*{remark*}{Remark}
\newtheorem{theo}{Theorem}[section]
\newtheorem{cor}{Corollary}[theo]

\newcommand{\X}{\mathcal{X}}
\newcommand{\XS}{\mathcal{X}_\mathcal{S}}
\newcommand{\Pro}{\mathbb{P}}
\newcommand{\Qro}{\mathbb{Q}}
\newcommand{\Sam}{\mathbb{S}}
\newcommand{\R}{\mathbb{R}}
\newcommand{\E}{\mathbb{E}}
\newcommand{\uh}{\widehat{u}}
\newcommand{\tauh}{\widehat{\tau}}
\newcommand{\tu}{u}
\newcommand\OX{\mathcal{O}_{\X}}
\newcommand\Ra{\mathcal{R}}
\newcommand\Up{\mathcal{U}}

\newcommand\tauhat{\widehat{\tau}}

\newcommand\AUUC{\operatorname{AUUC}}
\newcommand\AUNUC{\operatorname{AUNUC}}
\newcommand\AUNUCX{\operatorname{AUNUC^\X}}

\newcommand\VN{\operatorname{V_N}}
\newcommand\VNX{\operatorname{V_N^\X}}

\newcommand{\Voh}{\widehat{V}_1}
\newcommand{\Vth}{\widehat{V}_2}
\newcommand{\Vnh}{\widehat{V}_\nu}
\newcommand\D{\mathrm{d}}
\newcommand\one{\mathbbm{1}}

\newcommand\xp{x^\prime}
\newcommand\up{u^\prime}
\newcommand\rp{r^\prime}
\newcommand{\rew}[1][]{  
  \ifthenelse{\isempty{#1}}
    {\overrightarrow{y}}
    {\overrightarrow{y}^{(#1)}}
}

\newcolumntype{b}{>{\columncolor{magenta}}c}

\title{About Evaluation Metrics for Contextual Uplift Modeling}
\author{C. RENAUDIN, M. MARTIN}
\date{\today} 

\begin{document}

\begin{abstract}

In this tech report we discuss the evaluation problem of contextual uplift modeling from the causal inference point of view. More particularly, we instantiate the individual treatment effect (ITE) estimation, and its evaluation counterpart. First, we unify two well studied fields: the statistical ITE approach and its observational counterpart based on uplift study. Then we exhibit the problem of evaluation, based on previous work about ITE and uplift modeling. Moreover, we derive a new estimator for the uplift curve, built on logged bandit feedback dataset, that reduces its variance. We prove that AUUC fails on non randomized control trial (RCT) datasets and discuss some corrections and guidelines that should be kept in mind while using AUUC (re-balancing the population, local importance sampling using propensity score).
\end{abstract}
\maketitle
\tableofcontents
\section{Introduction}

Uplift modeling has been widely used to estimate the effect of a treatment on an outcome, at the user/individual level.
It has effectively been used in fields such as marketing and customer retention, to target
those customers that are most likely to respond due to the campaign or treatment.
This allows the selection of the subset of entities
for which the effect of a treatment will be "large enough" and, as such, allows the maximization of
the overall "reward". Specifically, it produces uplift scores which are used to essentially create a ranking between user inputs. From a budget constraint perspective, one aims to target first the customers with the highest "change in behavior" between being treated, and not treated.

In section \ref{sec:setting}, we first introduce the underlying statistical problem we aim at solving.
We formally define the AUUC maximization problem we want to solve, based on the fundamental problem of estimating treatment effect: for a given user we can only observe its outcome under the treatment he was facing.
ITE and Uplift modeling are very similar notions. They are formally introduced in section \ref{sec:setting}, as well as the assumptions guaranteeing that they are the same quantity. Another view of such a problem is through the lens of a two-armed bandit. Indeed, the ITE problem can be viewed as an instantiation of the off-line logged-bandit feedback problem, with binary action set $\{0,1 \}$. In order to identify ITE, one would need a full feedback for the outcome. The ITE can be reconstructed, based on the observed features, leveraging outcomes under both treatment and control.

In section \ref{sec:metrics} we define the usual metrics used to evaluate uplift models: PEHE and AUUC, and highlight both advantages and drawbacks. Specifically, we highlight AUUC problems: it should be used only for RCT studies and it needs special attention when the treatment group is not balanced (50\%-50\%). We generalize the theoretical AUUC quantity, based on the underlying distribution of the data.

We further present in Section \ref{sec:counter} some counter-examples, when the dataset is not RCT, i.e. $\Pro(t=1 \mid x)$  is not constant along features $x$, or if the treatment/control ratio of the dataset is not 50\%-50\%.

In Section \ref{section:unbalanced_dataset}, we introduce the generalized AUUC formula, and the "two axis" re-balancing to extend the AUUC metrics to non-RCT experiments.

Finally, in section \ref{sec:Vnu} we propose a new construction of the uplift curve, with new rules. We theoretically prove and empirically show that this new way of building the approximated uplift curve dramatically reduces the variance of derived AUUC.

\section{Related work}
This work is mainly built upon some recent works:
\begin{itemize}
    \item \cite{zhang2020unified} which presents the link between uplift and ITE.
    \item \cite{devriendt2020learning} which uses uplift methods in order to transfer them into the learning-to-rank community. This paper has the merit of highlighting the fact that in the uplift community, the metric that one should use is somehow messy. It summarizes it clearly in its main table.
    \item \cite{caron2020estimating} which recalls the different up to date uplift models that perform the best, and then test them on two semi-synthetic datasets. The only proposed metrics are ITE based, like PEHE. It does not mention AUUC.
    \item \cite{Gutierrez2016CausalIA} which presents a complete overview of uplift modeling with three main models, and the metrics commonly used for evaluation.
    \item \cite{brandfonbrener2020bandit} which exhibits the bandit error, when one only faces bandit type feedback. This work is mainly used for theory.
\end{itemize}

\section{Fundamental problem of uplift modeling}
\label{sec:setting}
\subsection{Potential outcome framework}
We consider in this report the problem based on the Rubin-Neyman potential outcome (PO) framework \cite{rubin1974estimating}, where subject's $i$ response/outcome with and without treatment are denoted respectively by $y_i^{(1)}$ and $y_i^{(0)}$. These variables are known as the potential outcomes. Moreover, in real life experiments, one can only inject user $i$ either a treatment (i.e. $t_i=1$) or no-treatment (i.e. $t_i=0$), and observe the response of the latter injection dose, (i.e. observes only $y_i^{(t_i)}$). The fundamental problem of counterfactual learning is based on the fact that one cannot observe the 2 potential outcomes $\rew_i=(y_i^{(1)},y_i^{(0)})$ for individual $i$. The hypothetical outcome $y_i^{(1-t_i)}$ is the unrealized
“counterfactual outcome” \cite{Bottou2013} and can only be estimated through a similar user $j \neq i$ that did receive the alternative treatment $t_j=1-t_i$. The notion of similarity shall be inferred based on \textit{observed} features (hereafter denoted by "features").

For instance, in medicine, the variable $y_i^{(t_i)}$ could model the fact that "subject $i$ has been cured after receiving the treatment $t_i$" where this boolean random variable (r.v.) is 1 if the patient took the drug, and 0 otherwise. In advertisement, one could model the fact subject $i$ has bought or not a product (boolean outcome variable $y_i^{(t_i)}$), after receiving a display (boolean condition $t_i \in \{ 0,1 \}$).

\subsection{Full information / Observable features}
We denote the characteristics space $\overline{\X} \subset \R^n$. For each user, a characteristic vector $\overline{x} \in \overline{\X}$ and a full feedback reward vector $\rew=(y^{(1)},y^{(0)}) \in \R^2$ are drawn from a joint distribution $(\overline{x},\rew) \sim \mathcal{D}$. We stress that $\rew$ may be dependent on $\overline{x}$ as they are jointly distributed. Specifically user $i$'s behavior is endowed in a $n$-dimensional characteristics vectors $\overline{x}_i \in \overline{\X}$. User \textit{full information} characteristics $\overline{x}_i$ fully describes the exact behavior (i.e. the outcome distribution) of user $i$.

In the logged-bandit feedback, one only "sees" the value of $y^{(t)}$, meaning only half of the full feedback reward vector $\rew$ is revealed, whereas in the "full feedback" problem, the whole reward $\rew$ is revealed. This "full feedback" is unrealistic, but it can be seen as the ideal dataset one could face to learn such treatment causal effect. The PO distribution $\left(y^{(1)}_ i, y^{(0)}_i\right)\sim \mathcal{D}_{\rew|\overline{x}}(\cdot|\overline{x_i})$, with the distribution $\mathcal{D}_{\rew|\overline{x}}$ being the marginal of $\mathcal{D}$.
In particular for each user $i$, we assume there exists a r.v. 
$$\left(y_i^{(1)},y_i^{(0)}\right)=\overrightarrow{y_i}: \Omega \rightarrow \R^2,$$ 
defined on the complete probability space $(\Omega, \mathcal{F}, \mathbb{P})$. This r.v. can be viewed as the underlying  outcome distribution of user $i$, if he/she has been treated or not.

In practice, real life constraints (ability to measure and/or privacy-preserving regulation) limit the amount of accessible information for user $i$; i.e. \textit{full information} characteristics $\overline{x}_i$ is only known through a function $g_{\text{obs.}}$. Specifically one only observes a feature vector $x_i=g_{\text{obs.}}(\overline{x}_i)\in \R^d$, where the function $g_{\text{obs.}}$ is the anonymization mapping into the  \textit{observable} feature space $\X$ we have access to: for instance, part of the browsing history accessible by an real time bidding (RTB) actor. From the recent privacy-preserving regulation, with related Turtledove\footnote{https://github.com/WICG/turtledove} and Sparrow\footnote{https://github.com/WICG/sparrow} proposals, the \textit{observable} features $x_i$ might be the cohort $j \in  [J]:=\{ 1, \dots, J \}$ user $i$ belongs to, i.e. $g_{\text{obs.}}(\overline{x}_i) \in  [J]$.

Typically the dimension of both feature space $\X$ and characteristics space $\overline{\X}$ satisfy $d \ll n$. For example, the \textit{observed features} used by advertisers to process inference (online recommendation system, etc), captures a small set of the user's behavioral attributes such as demographic information or part of its browsing history. 

The key about privacy preserving design is how well the function $g_{\text{obs.}}$ captures the original complete user information $\overline{x}_i$. The complete spectrum of such anonymization function $g_{\text{obs.}}$ goes from the identity mapping, meaning a full leakage of user privacy, to a uniform random assignment, in the cohort setting for instance where user privacy is guaranteed, i.e. $\forall x \in \X$, $g_{\text{obs.}}(\overline{x})\sim \mathcal{U}([J])$. The whole story of privacy preserving design is to find the right trade-off between these two extremes, guaranteeing a minimum of causal effect between $g_{\text{obs.}}(\overline{x})$ and the associated reward vector $\rew$ while ensuring $g_{\text{obs.}}$ preserves the user's privacy.

\begin{remark}
An interesting question is how to estimate the causal effect between the 2 variables $g_{\text{obs.}}(\overline{x})$ and $\rew$ for a given function $g_{\text{obs.}}$ (potentially designed from an adversarial manner), as a correlation would potentially not be sufficient and an intervention would not be possible, at least not directly on the variable $g_{\text{obs.}}(\overline{x})$. One could wonder, assuming there exists some "causal link" between $\overline{x}$ and $\rew$, whether a correlation between $g_{\text{obs.}}(\overline{x})$ and $\rew$ means a causal link, for every function $g_{\text{obs.}}$.
\end{remark}

\begin{remark}
An interesting metric to compute while anonymizing is the \textit{information loss} \cite{Agrawal2001Privacy, Shah2016} that quantifies how much information we lose from the whole vector of characteristics $\overline{x}_i$, while guaranteeing \textit{some} minimal performance for prediction. The story is usually set the other way around. The law imposes some minimal privacy guarantee for each user, and based on such a constraint, companies try to optimize the performance of their predictive algorithms (see Figure \ref{fig:priv+preserving}).
\end{remark}

\begin{figure}[!ht]
\centering
\includegraphics[width=0.7\textwidth]{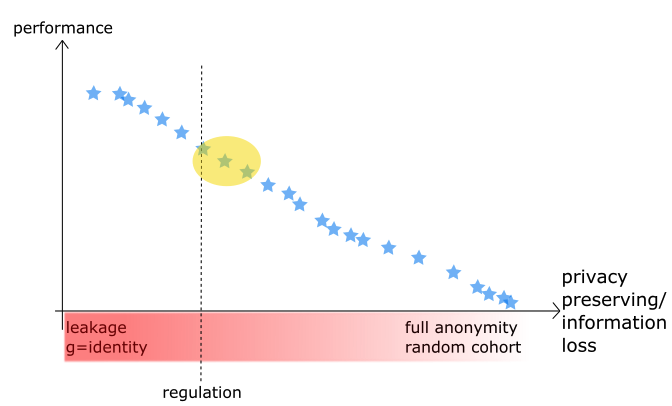}
\caption{Maximizing performance under privacy preserving regulation}
\label{fig:priv+preserving}
\end{figure}

\subsection{Treatment assignment}

On action set $\mathcal{T}$, a treatment assignment policy $\pi: \X \rightarrow \mathcal{P}(\mathcal{T})$ maps each feature (aka "context") $x \in \X$ to a distribution over the action set $\mathcal{T}$, i.e. the "treatment assignment". Specifically, $\pi(\cdot|x)$ is a distribution over the action set $\mathcal{T}$.  In the simple setting with binary treatment $\mathcal{T}=\{ 0,1 \}$ and binary reward $\rew_i \in \{0,1\}^2$, $\pi(\cdot|x)$ is often a Bernoulli distribution that can be summarized by the contextual value $\Pro(t=1|x)$, $\forall x \in \X$. A particular case is when the treatment assignment is stochastic, and \textit{does not depend} on the feature $x$. This is called Randomized Control Trial (RCT). 
\begin{defi}[RCT]
A treatment policy $\pi(\cdot \mid x)$ is said to be RCT, if and only if there exists a distribution $\pi_{RCT}$ over the action set, independent of $x$ and such that $\forall x \in \X$ $\pi(\cdot \mid x)=\pi_{RCT}(\cdot)$. In particular for a binary action set, the r.v. $t$ sampled using $\pi$ is such that $\Pro(t=1|x)=\alpha$, $\forall x \in \X$, and $\alpha$ being a constant.
\end{defi}

Although being sometimes expensive or unethical, RCT is the gold-standard for studying the causal relationship between a treatment and an outcome. Randomization eliminates the bias inherent to observational studies, where surprising phenomena like Simpson's paradox might occur \cite[Chapter 6.3]{PetersJanzingSchlkopfElements}.

\subsection{Logged-bandit versus full feedback}
We define here the two types of dataset we consider hereafter, namely the logged-bandit feedback one, and the full feedback one. 

The full feedback dataset $D_N^{\text{full}}$ is composed by the feature vector and the reward vector. 
\begin{equation*}
    D_N^{\text{full}}=\left\{ x_i, \overrightarrow{y_i} \right\}_{i=1}^N, \quad \overrightarrow{y_i} \in \R^2
\end{equation*}
No policy is required at this point, and any treatment policy could be plugged in afterward.

On the other hand, the logged-bandit $D_N$ (aka observational dataset) of size $N$ is a list of instances:
\begin{equation*}
    D_N=\left\{ x_i, t_i, y_i, q_i \right\}_{i=1}^N, \quad \text{with} \quad q_i=\pi(t_i|x_i)\quad \text{and} \quad y_i=\rew[t_i]_i.
\end{equation*}

\begin{remark}
N.B.: In many observational datasets, particularly in causal inference, we do not have access to the point-wise probability of being treated $q_i$. This is then estimated through a nuisance model, in order to fit the propensity score. This is not part of the present work, where we assume to have access to such $q_i$. 
\end{remark}

Note that the terminology \textit{observational dataset} should not be mismatched with \textit{observational studies}.
While the first covers both RCT and observational studies, the second only covers the process to assign treatment from a non random process, i.e. non-RCT.

\subsection{Underlying distribution}
Following Pearl's formalism \cite{PearlBookWhy}, one can propose an associated Structural Causal Model (SCM) in Figure \ref{figure:SCM}:
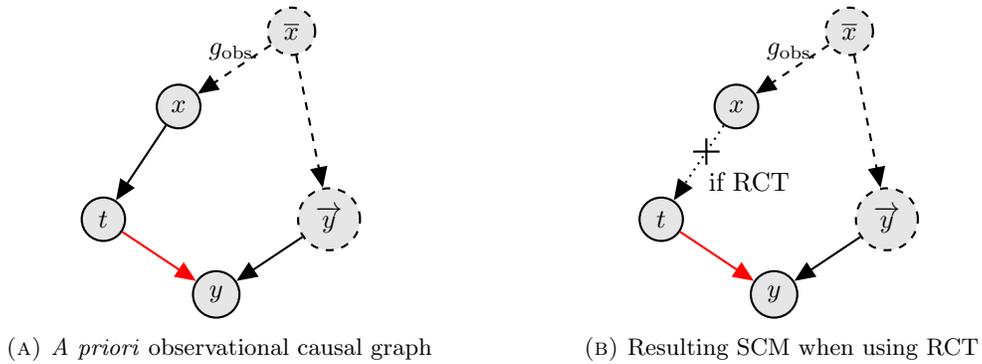
\begin{figure}[!ht]
     \centering

\begin{subfigure}{.45\textwidth}
\centering
\begin{tikzpicture}[auto, thick, node distance=1cm, >=triangle 45]
\tikzstyle{unobserved}=[thick, dashed, fill=gray!20]
\tikzstyle{norn}=[thick, fill=gray!20]
\tikzset{
    cross/.pic = {
    \draw[rotate = 45] (-####1,0) -- (####1,0);
    \draw[rotate = 45] (0,-####1) -- (0, ####1);
    }
}
\draw 
    node[norn] at (0,0)[causalvar](T){$t$}
    node[unobserved] at (3,0)[causalvar](Y){$\rew$}
    node[norn] at (1,1.5)[causalvar](X){$x$}
    node[norn] at (1.5,-1)[causalvar](y){$y$}
    node[unobserved] at (2.5,2.5)[causalvar](U){$\overline{x}$};
    \draw[red,->](T) -- node {} (y);
    \draw[->](Y) -- node {} (y);
    \draw[dashed,->](U) -- node[above] {$g_{\text{obs.}}$} (X);
    \draw[dashed,->](U) -- node {} (Y);
    \draw[->](X) -- node {} (T);
\end{tikzpicture}
\caption{\textit{A priori} observational causal graph}
\label{figure:SCM0}
\end{subfigure}
\hfill
\begin{subfigure}{.45\textwidth}
\centering
{\begin{tikzpicture}[auto, thick, node distance=1cm, >=triangle 45]
\tikzstyle{unobserved}=[thick, dashed, fill=gray!20]
\tikzstyle{norn}=[thick, fill=gray!20]
\tikzset{
    cross/.pic = {
    \draw[rotate = 45] (-#1,0) -- (#1,0);
    \draw[rotate = 45] (0,-#1) -- (0, #1);
    }
}
\draw 
    node[norn] at (0,0)[causalvar](T){$t$}
    node[unobserved] at (3,0)[causalvar](Y){$\rew$}
    node[norn] at (1,1.5)[causalvar](X){$x$}
    node[norn] at (1.5,-1)[causalvar](y){$y$}
    node[unobserved] at (2.5,2.5)[causalvar](U){$\overline{x}$};
	\draw[red,->](T) -- node {} (y);
	\draw[->](Y) -- node {} (y);
	\draw[dashed,->](U) -- node[above] {$g_{\text{obs.}}$} (X);
	\draw[dashed,->](U) -- node {} (Y);
	\draw[dotted,->](X) -- node {\text{if RCT}} (T);
	\draw (.6,.9) pic[rotate = 45] {cross=5pt};
\end{tikzpicture}
}
\caption{Resulting SCM when using RCT}
\label{figure:SCM1}        
\end{subfigure}
\caption{Causal graph $\mathcal{G}_{SCM}$ induced by the SCM}
\label{figure:SCM} 
\end{figure}

The arrow from $x$ to $t$ is removed in the case of an RCT experiment. This allows to estimate the causal effect of the treatment, i.e. the red arrow from $t$ to $y$ in Figure \ref{figure:SCM}.

If $t \in \{ 0,1 \}$, we define the expected reward for (full information) characteristics $\overline{x} \in \overline{\X}$, under treatment $t$, by
\begin{equation} \label{eq:true_conv}
    \overline{p_t}(\overline{x})=\mathbb{E}(y | do(t), \overline{x}),
\end{equation}
where the $do(\cdot)$ is formally defined in Pearl's book \cite{PearlBookWhy}.
In particular, if the outcome $y$ is binary, equation \eqref{eq:true_conv} becomes:
\begin{equation}
    \overline{p_t}(\overline{x})=\mathbb{P}(y=1 | do(t), \overline{x}).
\end{equation}
Similarly, one can define the same quantity on observable features $x \in \X$ by:
\begin{equation} 
    {p_t}({x})=\mathbb{E}(y | do(t), {x}),
\end{equation}
and if the outcome is binary:
\begin{equation}
    {p_t}({x})=\mathbb{P}(y=1 | do(t), {x}).
\end{equation}



\subsection{ITE}
The Individual Treatment Effect (ITE) is a r.v. defined by the difference between the two PO:
$$\tau=\left\langle \rew , \begin{bmatrix}
         1 \\
         -1
        \end{bmatrix} \right\rangle=\rew[1]-\rew[0].$$ 
Specifically for subject $i$, the ITE $\tau_i$ is the incremental benefit of being treated, versus not:
\begin{equation*}
    \tau_i:=\rew[1]_i-\rew[0]_i.
\end{equation*}
As explained before, in the real world, one can never observe at the same time both quantities $\rew[1]_i$ and $\rew[0]_i$, and then the ITE $\tau_i$ is not \textit{identifiable}. A generalization trick is based on estimating the so called Conditional Average Treatment Effect (CATE). The CATE is based on the observable features $x$ only, and generalizes around similar $x$ to reconstruct such ITE. Specifically, we define the CATE $\tau(x)$ by:
\begin{equation} \label{eq:ite}
\tau(x):= \E[\tau|x]=\E[\rew[1]-\rew[0]|x]. 
\end{equation}

Similarly, if we consider an RCT assignment, and if the r.v. $y=t \rew[1]+(1-t)\rew[0]$, then the uplift is defined as follows:
\begin{equation} \label{eq:upl}
    \tu(x):=\E[y|t=1, x]-\E[y|t=0, x]=p_1(x)-p_0(x).
\end{equation}

\subsection{Linking the treatment effect heterogeneity and uplift objectives}
\label{sec:ITE=uplift}
Based on \cite{zhang2020unified} we can decompose the CATE as:
\begin{align}
\begin{split}
\tau({x})=& \mathbb{E}[\rew[1]-\rew[0] \mid {x}] \\
=& \mathbb{E}[\rew[1] \mid t=1, {x}] P(t=1)+\mathbb{E}[\rew[1] \mid t=0, {x}] P(t=0) \\
&-\mathbb{E}(\rew[0] \mid t=1, {x}) P(t=1)-\mathbb{E}[\rew[0] \mid t=0, {x}] P(t=0) \\
=& \underbrace{\mathbb{E}[\rew[1] \mid t=1, {x}]-\mathbb{E}[\rew[0] \mid t=0, {x}]}_{\text {observed }} \\
&+P(t=1)\{\underbrace{\mathbb{E}[\rew[0] \mid t=0, {x}]}_{\text {observed }}-\underbrace{\mathbb{E}[\rew[0] \mid t=1, {x}]\}}_{\text {unobserved }}\\
&+P(t=0)\{\underbrace{\mathbb{E}[\rew[1] \mid t=0, {x}]}_{\text {unobserved }}-\underbrace{\mathbb{E}[\rew[1] \mid t=1, {x}]\}}_{\text {observed }}
\end{split}
\end{align}
The last equality splits the CATE into 3 terms, hereafter deeply explained under three common assumptions.

The first term is equal to the uplift term $\tu(x)$ defined in equation \eqref{eq:upl}, because the PO equals the observing outcome when conditioning on $t$. The second and third terms are null under the following unconfoundedness.
\begin{ass}[Unconfoundedness, aka strong ignorability]
\label{ass:unconfoundedness}
The potential outcome couple is independent of the treatment $t$, when conditioning on the observed co-variate variables $x$, i.e.
\begin{equation}
\left\{y^{(0)}, y^{(1)}\right\} \perp t \mid x
\end{equation}
\end{ass}
Under Assumption \ref{ass:unconfoundedness}, one immediately gets:
\[ \E[y^{(0)} \mid t=0, x]=\E[y^{(0)} \mid x], \]
leading to:
\[ \E[y^{(0)} \mid t=0, x]-\E[y^{(0)} \mid t=1, x]=\E[y^{(0)} \mid x]-\E[y^{(0)} \mid x]=0. \]
The estimation of the uplift term \eqref{eq:upl} only involves observational data without any counterfactual. It can be estimated with no bias guaranteed, under the 2 following conditions: overlapping and SUTVA.

\begin{ass}[Overlapping treatment]
\label{ass:overlap}
Any subject $i$ has a non-zero probability of receiving treatment and control, i.e. for every $x \in \R^d$,
\begin{equation}
0<\Pro(t=1\mid x)<1
\end{equation}
\end{ass}
\begin{remark}
In assumption \ref{ass:overlap}, we somehow extend the RCT setting, as the propensity score $\Pro(t=1\mid x)$ may depend on $x$. Assumption \ref{ass:overlap} remains true in RCT experiments if $\epsilon \leq \pi_{\text{RCT}}(\cdot) \leq 1-\epsilon$ for some $\epsilon >0$.
\end{remark} 

\begin{ass}[Stable Unit Treatment Values Assumption]
\label{ass:sutva}
A subject's potential outcome is not affected by other subjects treatment. In other words, treatment applied to one subject does not affect the outcome of other subjects.
\end{ass}

As explained in \cite{zhang2020unified}, Assumption \ref{ass:sutva} often applies in health related application. But, in online advertising, because of the wide existence of crowd-sourced coupon sharing websites, this assumption may fail. Indeed not treating a user does not guarantee that this user will effectively not be treated by a competitor.

From a same perspective, Assumption \ref{ass:overlap} requires to maintain a minimum exploration level for every feature. It might be the case for low traffic/short period experiments, but certainly not from a production machinery in advertising for instance.

\begin{remark}
In many papers about uplift modeling, it is implicitly assumed that the data comes from RCT experiments or A/B tests \cite{Gutierrez2016CausalIA}. In particular, RCT implies the equivalence between conditioning and intervening \cite{PearlBookWhy}, i.e. $t=1 \Leftrightarrow \text{do}(t=1)$, because $T$ has no parents in the associated SCM, and $\emptyset$ is a valid adjustment set for $(T,Y)$ \cite{PetersJanzingSchlkopfElements}.
\end{remark}

\begin{remark}
Assumption \ref{ass:unconfoundedness} is not testable from dataset $D_N$ only. In short, we want to check if the projected observable feature information $x$ is enough in order to fully explain the heterogeneity in treatment, and ultimately, to build uplift models. An easy answer has been proposed in \cite{li2020general} and requires that $x$ contains all direct causes of the outcome $y$ and contains no effect variables of $y$, which is usually not verified in advertising, considering a sequential timeline: as the available feature is very limited ($d \ll n$), the variable $y_t$ may influence future feature $x_{t+1}$. Assumption \ref{ass:unconfoundedness} fails if this feedback loop phenomenon happens.
\end{remark}

\subsection{Modeling uplift in practice}
Literature has increased fast about building more and more precise models in order to fit at best $\tau$. This is beyond the scope of this tech report, but a recent survey \cite{zhang2020unified} clearly explains all subtleties of each standard method.
A general model fitting the CATE $x \mapsto \tau(x)$ of Uplift $x \mapsto \tu(x)$ will be hereafter denoted by $\tauhat$ and $\uh$.




\section{Metrics for evaluation}
\label{sec:metrics}
\subsection{Metrics with known ground truth}
For synthetic or semi-synthetic datasets, as we know the generating model for $y \mid t,x$, we can use the ground truth to evaluate the learned model. The most popular metric to do so is the Precision in Estimation of
Heterogeneous Effects (PEHE) (Hill, 2011), which computes the MSE between the estimated CATE $\tauhat$ and the exact one using the ground truth $\tau$. Formally, PEHE is defined as
\[ \text{PEHE}=\frac{1}{N}\sum_{i=1}^{N} \left(\tauhat(x_i)-\tau(x_i) \right)^2\]

\subsection{Metrics without ground truth}
\subsubsection{Qini and Uplift curves}
As recall in \cite{devriendt2020learning}, when the true CATE $\tau$ is not known, the usual principle of uplift modeling is to find the best way of ranking instances by their respective uplift values $\tauhat$. Specifically, \cite{Radcliffe2012RealWorldUM} suggests to visualize the so called \textit{incremental expected uplift value} for an incrementally larger subgroup of the ranked population. For example, for the top 10\%, 20\%, ..., 100\% of the instances.

\begin{figure}[!ht]
\centering
\includegraphics[width=0.5\textwidth]{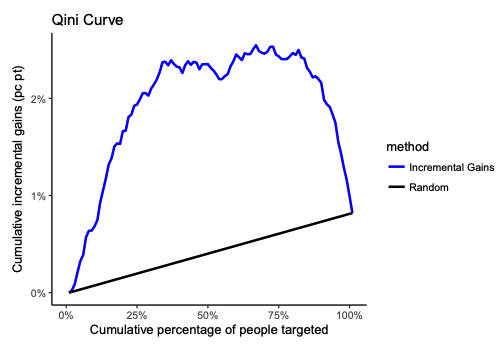}
\caption{Incremental gain curve (blue) versus the expected theoretical gain using random targeting (black)}
\label{fig:AUUC_ex}
\end{figure}

Figure \ref{fig:AUUC_ex} shows such an \textit{incremental expected uplift value}: the blue line represents the cumulative incremental gains as a function of the selected fraction of the ranked population, while the black line represents the expected value of a random sub-sample of that size, called the random baseline. We expect a good uplift model to rank first the individuals likely to respond when treated, leading to higher estimated uplift values in the early parts of the plot. Different approaches have been proposed in the literature to do so, the two best known being the Qini Curve and the Uplift Curve. We stress that there is a further difference
between the curves themselves: which values are plotted on the $y$-axis and $x$-axis, and which area under the curve is returned (with or without subtracting the random line).

As noticed by the authors of \cite{devriendt2020learning}, we cannot find unique definitions for Qini and Uplift curves in the broad literature. Nevertheless, we can split all these definitions into 2 categories:
\begin{itemize}
    \item \textit{separated} versus \textit{joint} ranking: is the ranking computed on treatment and control groups \textit{separately} or as one \textit{joint} group?
    \item \textit{absolute} versus \textit{relative} incremental gains: are the incremental gains computed as absolute values or as values relative to the number of individuals?
\end{itemize}

The second difference (\textit{absolute} versus \textit{relative}) is actually related to the special care required to compute uplift metrics when the dataset is unbalanced (i.e. when the treatment/control split ratio is different than 50\%-50\%). They require a re-balancing that can be \textit{partially} (this re-balancing along the $y$-axis is developed in Section \ref{sec:defvariants}) implemented with relative incremental gains or with absolute incremental gains associated with re-scaling factors.

We want to highlight here a third difference among variants in literature, that is still related to the re-balancing topic, but is orthogonal to the second difference (this third difference is not introduced as such in \cite{devriendt2020learning}, but presented through examples): the re-balancing can be implemented with global factors (based on the sizes of the complete treatment and control groups) or with current estimated factors (based on the sizes of the incremental treatment and control subgroups defined by the first $k$ elements).

\subsubsection{Definition of some variants}
\label{sec:defvariants}
Let $\uh$ be a general uplift model thought of as an approximation of the true uplift $\tu$.

To formalize the \textit{joint} variants, we first denote by $\Upsilon(D_N, k)$ the decreasing order of the dataset $D_N$ under $\uh(\cdot)$. Specifically $\Upsilon(D_N,k)$ represents the first $k$ instances in $D_N$ according to the ordering induced through the score $\uh(\cdot)$, i.e.
\begin{align} 
\Upsilon(D_N, k) &=\left\{\left(x_{i}, t_i, y_{i}\right) \in D_N\right\}_{i=1, \ldots, k} \text { such that } \forall i \leq k, \forall k<l,  \, \uh\left(x_{i}\right) \geq \uh\left(x_{l}\right)
\end{align}
We define the permutation of $[N]$ $\phi \in \mathfrak{S}_N$ that sorts the dataset instances by \textit{decreasing} uplift model value $\uh(x_i)$ (with ties arbitrarily split), i.e. let $\phi \in \mathfrak{S}_N$ s.t.
\begin{equation}
\label{eq:perm}
i<j \Rightarrow \uh(x_{\phi(i)}) \geq \uh(x_{\phi(j)}).    
\end{equation}

Among these top-$k$ ranked individuals, we define
\begin{align}
N_{\Upsilon}^{T}(D_N, k)&=\sum_{\left(x_{i}, t_i, y_{i}\right) \in \Upsilon(D_N, k)}\one_{t_{i}=1}\\
N_{\Upsilon}^{C}(D_N, k)&=\sum_{\left(x_{i}, t_i, y_{i}\right) \in \Upsilon(D_N, k)}\one_{t_{i}=0}
\end{align}
and the number of responders among the top-$k$ individuals are defined as:
\begin{align}
R_{\Upsilon}^{T}(D_N, k)&=\sum_{\left(x_{i}, t_i, y_i\right) \in \Upsilon(D_N, k)}\one_{y_i=1}\one_{t_i=1}\\
R_{\Upsilon}^{C}(D_N, k)&=\sum_{\left(x_{i}, t_i, y_i\right) \in \Upsilon(D_N, k)}\one_{y_i=1}\one_{t_i=0}
\end{align}

To formalize the \textit{separated} variants, we consider the treatment group $T=\{(x, t, y) \in D_N \mid t=1\}$ and the control group $C=\{(x, t, y) \in D_N \mid t=0\}$, and we denote by $\Upsilon(T, k)$ (resp. $\Upsilon(C, k)$) the decreasing order of the treatment group (resp. control group) under $\uh(\cdot)$. We define the number of responders among the top-k individuals on each group as:
\begin{align}
R_{\Upsilon}(T, k)&=\sum_{\left(x_{i}, y_{i}\right) \in \Upsilon(T, k)}\one_{y_i=1}\\
R_{\Upsilon}(C, k)&=\sum_{\left(x_{i}, y_{i}\right) \in \Upsilon(C, k)}\one_{y_i=1}
\end{align}

The Qini and Uplift curves $V(\cdot)$ can now be defined as a function of $k \in \{1...N\}$ (\textit{joint} setting) or $p \in [0, 1]$ (\textit{separated} setting), where $p$ is the proportion for both treatment and control groups (e.g. 1\%, 2\%, 3\%, ...). Table \ref{table:curvedefs} from \cite{devriendt2020learning} lists several formulas for $V(\cdot)$ found in literature.

\begin{centering}
\begin{table}[!ht]
\begin{center}
\makebox[\textwidth]{
\begin{tabular}{ |c|c|c|c| } 
 \hline
 Rank & Count & Qini Curve & Uplift Curve \\ 
  \hline
    \multirow{2}{2em}{Sep.} & Abs. & $V(p)=R_{\Upsilon}(T, p|T|)-R_{\Upsilon}(C, p|C|) \frac{|T|}{|C|}$ & $V(p)=R_{\Upsilon}(T, p|T|)-R_{\Upsilon}(C, p|C|)$ \\ 
 \cline{2-4}
  & Rel. & & $V(p)=\frac{R_{\Upsilon}(T, p|T|)}{|T|}-\frac{R_{\Upsilon}(C, p|C|)}{|C|}$ \\ 
 \hline
 \multirow{2}{2em}{Joint.} & Abs. & $V(k)=R_{\Upsilon}^{T}(D_N, k)-R_{\Upsilon}^{C}(D_N, k) \frac{N_{\Upsilon}^{T}(D_N, k)}{N_{\Upsilon}^{C}(D_N, k)}$ & $V(k)=\left(\frac{R_{\Upsilon}^{T}(D_N, k)}{N_{\Upsilon}^{T}(D_N, k)}-\frac{R_{\Upsilon}^{C}(D_N, k)}{N_{\Upsilon}^{C}(D_N, k)}\right) \left(N_{\Upsilon}^{T}(D_N, k)+N_{\Upsilon}^{C}(D_N, k)\right)$ \\ 
 \cline{2-4}
  & Rel. & \multicolumn{2}{c}{$V(k)=\frac{R_{\Upsilon}^{T}(D_N, k)}{|T|}-\frac{R_{\Upsilon}^{C}(D_N, k)}{|C|}$} \\ 
 \hline
\end{tabular}}
\caption{
Evaluation measures for uplift modeling. Two main approaches are considered, the Qini Curve and the Uplift Curve, both over two dimensions: ranking the data separately per group or jointly over all data, and expressing the volumes in absolute or relative numbers. From \cite{devriendt2020learning}}.
\label{table:curvedefs}
\end{center}
\end{table}
\end{centering}

Here, we want to introduce a new joint/absolute case variant for the uplift curve, including a local Importance Sampling (IPS) \cite{Cochran77, mcbook}
correction:
\begin{equation} \label{eq:local_avg}
V_{\text{IPS}}(k)=\frac{R_{\Upsilon}^{T}(D_N, k)}{e_T(k)}-\frac{R_{\Upsilon}^{C}(D_N, k)}{1-e_T(k)},    
\end{equation} 
where $e_T(k)$ is a local averaging of the probability to be treated, around the $k$-th entry. We could compute it based on a local kernel averaging, by $e_T(k)=\sum_{k'\in \mathrm{Z}} \text{ker}(k-k')\one_{t_k'=1}$, with a generic invariant kernel function\footnote{in particular, a kernel function is normalized, i.e. $\sum_{k'\in \mathrm{Z}} \text{ker}(k')$=1.} $x \mapsto \text{ker}(x)$. 

Another approach is by using the standard IPS formula, considering $\pi$ being the gold standard RCT 50\% treatment policy, and if $\pi_0$ denotes the logging policy. One can write: \begin{equation}
\widehat{V}_{\text{IPS}}(k)=\E_{\pi_0}\left[{\frac{\pi}{\pi_0}v_k}\right]    
\end{equation} 
where $v_k=\frac{1}{N} \sum_{i=0}^k \mathbbm{1}\left(t_{\phi(i)}=1 \wedge y_{\phi(i)}=1\right)-\mathbbm{1}\left(t_{\phi(i)}=0 \wedge y_{\phi(i)}=1\right)$ with $\phi$ defined in \eqref{eq:perm}.

Note that in Table \ref{table:curvedefs}, only the Uplift/Sep./Rel. formula does not actually re-balance the treatment and control populations. The diversity of all the other formulas should not hide their similarity.

Imagine a theoretical dataset generated by a perfect RCT process. In such an unrealistic dataset, whatever the uplift interval you select, the proportion of treatment (resp. control) individuals within this interval is equal to the total proportion of treatment (resp. control) individuals, i.e.
$$\forall \, p \in [0, 1] \quad \Upsilon(D_N, k) = \Upsilon(T, p|T|) \cup \Upsilon(C, p|C|) \quad \text{with} \quad k = p(|T| + |C|)$$

In particular, this implies that: $N_{\Upsilon}^T(D_N, k) = p|T|$, $N_{\Upsilon}^C(D_N, k) = p|C|$, $R_{\Upsilon}^T(D_N, k) = R_{\Upsilon}(T, p|T|)$, $R_{\Upsilon}^C(D_N, k) = R_{\Upsilon}(T, p|C|)$.

With this setup, all $V(p)$ and $V(k)$ variants (except Uplift/Sep./Rel.) presented in Table \ref{table:curvedefs} are proportional with respect to a constant that depends on the re-balancing method.

\subsubsection{Re-balancing along the two axes}
\label{sec:normtwoaxes}
All the previously defined $V$ formulas (except Uplift/Sep./Rel.) address the re-balancing along the $y$-axis. But when calculating the area under the Qini/Uplift curve we also have to re-balance along the $x$-axis, in particular in the \textit{joint} setting. This is not clearly stated in \cite{devriendt2020learning} that focuses on the integration of curves in the \textit{separated} setting. Indeed, integrating with $p$ over $[0, 1]$ implicitly re-balances along the $x$-axis because the width of any interval $[0, p]$ does not depend on the actual proportion of treatment and control individuals it represents.
In the \textit{joint} setting, one may want to compute the AUUC by $\sum_{k=1}^N{V(k)}$. In the case of an unbalanced dataset, this formula underestimates the contribution of the minority group, even if $V(k)$ implements a re-balancing along the $y$-axis. We address this case in section \ref{section:unbalanced_dataset}.

\subsection{In practice, how to select the best Uplift model?}
\subsubsection{Pragmatic criteria to select an uplift model}
We state here the 6 criteria by Radcliffe in his pioneer work \cite{Radcliffe2012RealWorldUM} when one wants to choose the "best" uplift model among several models learned on the same dataset:

\begin{itemize}
\item "Validated qini": the highest qini/AUUC value.

\item "Monotonicity of incremental gains": intuitively, after sorting the individuals by decreasing estimated uplift score, each segment $[p, p+\delta]$ with $p \in [0,1]$ should have an average slope "smaller" than the previous one, producing a concave-like uplift curve.

\item "Maximum impact": how much extra positive outcomes the model predicts at its peak.

\item "Impact at cutoff": the cumulative uplift at a given percentile threshold.

\item "Tight validation": is the estimated curve on validation set similar to the estimated curve on the training dataset?

\item "Range of predictions.": models with a larger prediction range are supposed to be more useful.
\end{itemize}

\subsubsection{Practical questions}
This raises the question of how we can choose among uplift models if their AUUC/Qini is equal but their uplift curves have different shapes. It depends on the quantity of interest (QoI) and the budget constraints we face. Usually the optimization problem writes:
\begin{equation}
    \max_{a\leq \text{constr}  \leq b} \text{QoI}
\end{equation}
\noindent with 
\begin{itemize}
    \item $\text{QoI}  \in \left\{ \text{incremental user}, \text{incremental gain}, \text{RoI}, \dots  \right\}$,
    \item $\text{constr}  \in \left\{ \text{RoI},  \text{\# targeted users} \approx \text{budget}, \dots \right\}$
\end{itemize} 
with $ \text{RoI}=\frac{\text{\#incremental user }}{\text{\#treated user}} \approx $ slopes  on Figure \ref{fig:sameAUUC_detail}.

Figures \ref{fig:sameAUUC} instantiate such a problem, designing three uplift curves such that their AUUC are strictly equal. The blue curve captures the maximum incremental gain, i.e. will be preferred in a low cost setting Figure \ref{fig:QoI1}, in particular, when we do not face any minimum RoI constraints. On the other hand the red curve has the best RoI at its peak (maximum incremental gain meaning optimal ratio treated) leading to the cheapest cost for each conversion (the cost being the inverse of the RoI, that reads by the slope between (0,0) and this point at the maximum of the uplift curve as shown on Figure \ref{fig:QoI2}). This uplift model is the \textit{best} if we are able to pay this RoI. This would happen when we want to \textit{maximize the RoI}, without constraints on the budget. Finally we face a very high RoI constraint, one may prefer the purple curve (first section of the uplift curve has the highest slope as shown on Figure \ref{fig:QoI3}). We may prefer switching to the red curve when progressively allowing lower RoI. If we continue to lower the minimum RoI bound, we end up choosing the blue curve as shown on Figure \ref{fig:QoI4}. All these possibilities cannot be summarized in the single AUUC metric, as this particular metric cannot distinguish between the latter production QoI and constraints (same AUUC). 
\begin{remark}
The previous examples consider concave curves only. The concavity of such uplift curves is fundamental, and in the estimation process one could try to draw the concave envelope of the realized uplift curves.
\end{remark}

\begin{figure}[!ht]
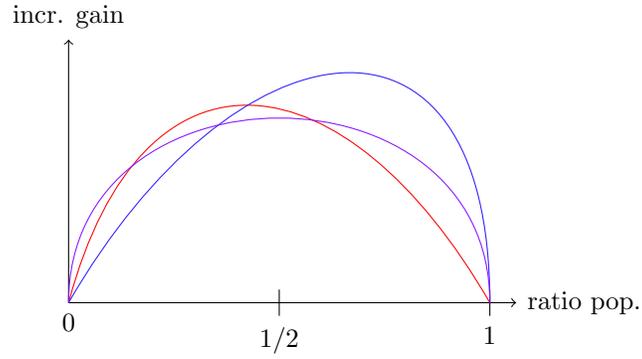

\ctikzfig{fig4}
\caption{Three uplift models with exact same AUUC}
\label{fig:sameAUUC}
\end{figure}

\begin{figure}[!ht]
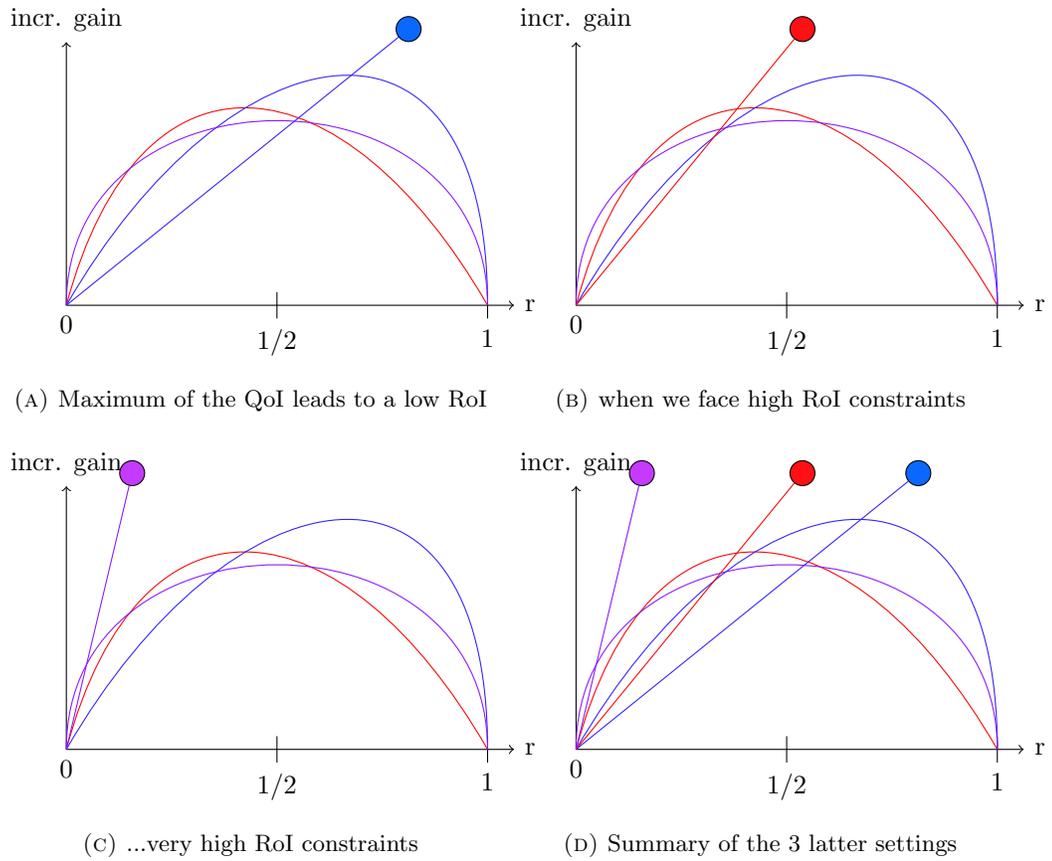

\centering
\begin{subfigure}[b]{0.49\textwidth}
\ctikzfig{fig41}
\caption{Maximum of the QoI leads to a low RoI}
\label{fig:QoI1}
\end{subfigure}
\begin{subfigure}[b]{0.49\textwidth}
\ctikzfig{fig42}
\caption{when we face high RoI constraints}
\label{fig:QoI2}
\end{subfigure}
\begin{subfigure}[b]{0.49\textwidth}
\ctikzfig{fig43}
\caption{...very high RoI constraints}
\label{fig:QoI3}
\end{subfigure}
\begin{subfigure}[b]{0.49\textwidth}
\ctikzfig{fig4all}
\caption{Summary of the 3 latter settings}
\label{fig:QoI4}
\end{subfigure}
\caption{Practical pathological uplift curves with same AUUC/Qini value}
\label{fig:sameAUUC_detail}
\end{figure}

\subsection{Theory behind the uplift curve and AUUC metric}
\label{section:theoryauuc}

\begin{defi}[`transferred' measure]
Given a $d$-dimensional valued r.v. $x \in \X \subset \R^d$ (user features) with associated probability measure $\Pro$, and an uplift model $\uh: \X \rightarrow \R$, we define the `transferred' measure $\Pro_2$ on $\R$ such that (see Figure \ref{fig:measure_map}) for any $\Pro$-measurable set $A \subset \R^d$
$$\Pro_2(A):=\uh_{\#}{\Pro}(A)=\Pro(\left(\uh\right) ^{-1}(A)).$$
\end{defi}

For example, for $\xi>0$, $\Pro_2([\xi, +\infty[)$ represents the probability that the model $\uh$ maps $x$ to a value higher than $\xi$.


\begin{figure}[!ht]
\centering

\ctikzfig{figDistr}
\caption{Transfer of the probability measure $(\X,\Pro)$ onto $(\R, \Pro_2)$}
\label{fig:measure_map}
\end{figure}

\begin{defi}[uplift curve]
\label{def:up_curve}
For an uplift model $\uh$ we define the Normalized Uplift Curve $\VN: [0,1] \rightarrow \R$ as:
$$\VN(r \in [0,1])= \int_\X \tau(x) \one_{[\uh(x)\geq \xi(r)]} \D\Pro(x) \quad \text{where} \quad
\xi(r) = \inf \{\xi \in \R: \Pro(x : \uh(x)\geq \xi)\leq r\}
.$$

Equivalently, we can define
$\xi(r) =\inf \{\xi \in \R: \Pro_2([\xi,+\infty[)\leq r\}$.
\end{defi}

Intuitively, $\xi(r)$ corresponds to the threshold such that a proportion $r$ of users have an uplift $\geq \xi(r)$.


We define the Area Under the Normalized Uplift Curve ($\AUNUC$) as:
\begin{equation}\label{auuc_def}
\AUNUC(p \in [0,1])=\int_{[0,p]} \VN(r) \D r
\end{equation}

For a given dataset $D_N$, $\AUUC$ differs from $\AUNUC$ only by a constant factor that depends on the dataset size (and the re-balancing method chosen for $V(p)$ to manage unbalanced datasets) but not on the evaluated uplift models themselves. Thus comparing uplift models with $\AUNUC$ is equivalent to comparing them with $\AUUC$, as long as the dataset is the same.

\subsection{Computing the AUNUC empirically}
\label{sec:aununc-emp}
Here we recall the way of approximating the AUUC, in the case where we face a logged-bandit dataset, i.e. we do not access the full feedback dataset, but only $\rew^{(t_i)}$ is given as reward. The steps are threefold:
\begin{itemize}
\item Given a logged-bandit feedback dataset $D_N=\left\{ x_i, t_i, y_i, q_i \right\}_{i=1}^N$, with $q_i=\pi(t_i|x_i)$ and $y_i=\rew[t_i]_i$, we argsort through the permutation $\phi \in \mathfrak{S}_N$ by \textit{decreasing} uplift model value $\uh(x_i)$ (with ties arbitrarily split), i.e. let $\phi \in \mathfrak{S}_N$ s.t.
$$i<j \Rightarrow \uh(x_{\phi(i)}) \geq \uh(x_{\phi(j)}).$$
\item Then, we convert the sorted units $i \in [N]$ into a piece-wise linear function globally continuous, following the rule: beginning from (0, 0), we go one unit $1/N$ right and
\begin{enumerate}
    \item we go $1/N$ up iff $t_{\phi(i)}=1 \text{ and } y_{\phi(i)}=1$ [green segments on Figure \ref{fig:build}], 
    \item we go $1/N$ down iff $t_{\phi(i)}=0 \text{ and } y_{\phi(i)}=1$ [red segments on Figure \ref{fig:build}], 
    \item we stay flat  otherwise [horizontal segments on Figure \ref{fig:build}]
\end{enumerate}
From a more rigorous mathematical standpoint, one can define the following quantity:
$$\widehat{V}(k):=\frac{1}{N}\sum_{i=1}^k \mathbbm{1}\left(t_{\phi(i)}=1 \wedge y_{\phi(i)}=1\right)-\mathbbm{1}\left(t_{\phi(i)}=0 \wedge y_{\phi(i)}=1\right)$$
and join each pair of successive points $\left(\frac{k}{N},\widehat{V}(k)\right)$, $\left(\frac{k+1}{N},\widehat{V}(k+1)\right)$ for $k \in [N-1]$. This produces the uplift curve.
\item Then, we join $(0,0)$ with the arrival point that represents a uniformly random treatment policy on the whole population (assigning i.i.d. random uplift values for each instance) [{\color{black} black semi dashed line on Figure \ref{fig:build}}]
\item Finally we compute the area between the constructed uplift curve and the latter described straight line. This is called the $\Delta$AUNUC. The area between the curves specifies the quality of the uplift model on the considered population.
\end{itemize}

\begin{figure}[!ht]
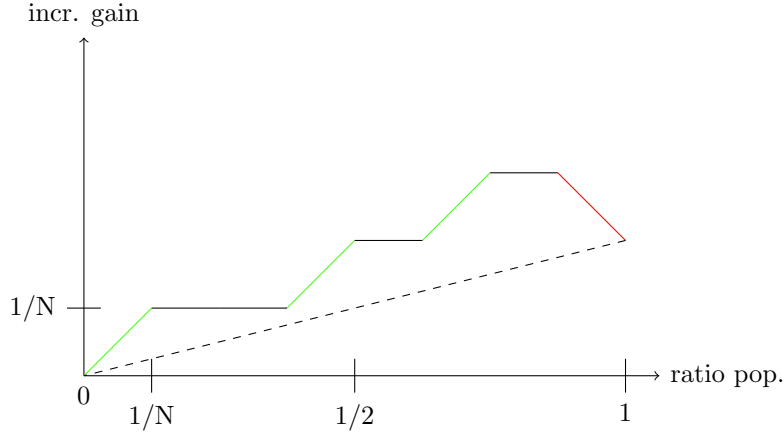

\centering
\ctikzfig{fig5}
\caption{Building AUUC}
\label{fig:build}
\end{figure}


\section{Toy examples on AU(N)UC fails}
\label{sec:counter}
\subsection{What does non-RCT means?}
First, let us recall the definition of RCT:\\
RCT: $ \exists\, \alpha \in ]0, 1[$, s.t. $\forall x \in \X$, $\Pro(T=1 \mid x)=\alpha$\\
Then the negation of the latter definition becomes:\\
non-RCT: $ \exists \, x_1,x_2 
\in \X$, $\Pro(T=1 \mid x_1)\neq\Pro(T=1 \mid x_2)$\\
The generalization of non RCT setting, based on neighborhoods, assuming that the function $x \mapsto \Pro(T=1 \mid x)$ is continuous writes:\\
non-RCT on neighborhood: $ \exists \, \X_1, \X_2 \subset \X \text{ with } \X_1 \cap \X_2 = \emptyset$, s.t. $\Pro(T=1 \mid \X_1)\neq\Pro(T=1 \mid \X_2)$

\subsection{Toy example 1: AUUC on an observational study}
\label{sec:toyexample1}
Let us build a toy example such that the AUUC metric does not properly rank a perfect model and a non perfect model, i.e. gives the highest score to the latter.

More particularly, let us imagine facing 4 groups of individuals/customers with the same size, i.e. each of them containing $N/4$ samples from a dataset of size $N$. The 4 particular groups reflect typical users \cite{beji2020estimating} and are named convincible (CO), Sure Thing (ST), Lost Cause (LC) and Sleeping Dogs (SD) such that: 
\begin{itemize}
    \item CO are users that give a positive output, if and only if they are treated. They symbolise where benefit of the treatment occurs. Specifically, $y^{1}=1$ and $y^{0}=0$.
    \item ST are users that always have positive outcome, whatever they are treated or not, i.e. with $y^{1}=1$ and $y^{0}=1$.
    \item LC are users that always have null outcome, whatever they are treated or not, i.e. with $y^{1}=0$ and $y^{0}=0$.
    \item SD are users with negative reaction w.r.t. treatment $t$, i.e. they have a positive outcome iif they are not treated. They symbolise the population where we aim at applying no treatment. Specifically, they are represented with the 2 PO's $y^{1}=0$ and $y^{0}=1$.
\end{itemize}

We consider noiseless rewards a.k.a. potential outcomes.

We present in Table \ref{counterexample:table} the population of each group, as well as $y^{1}$, $y^{0}$ (P.O.), $u$ being the true uplift, and another approximated uplift model $\uh$ that we will tune later through its constant values $\alpha_i$ across each group. We present, as well, the probability to be treated in each group, $\Pro(t=1|x)$, that is uniform within each group but does not have to be the same from one group to another (a particular case of non-RCT relaxed).

\begin{table}[!ht]
\begin{center}
\begin{tabular}{ | c || c| c | c | c | } 
\hline
groups $(x)$ & CO & ST & LC & SD \\ 
\hline \hline
pop. & $N/4$ & $N/4$ & $N/4$ & $N/4$ \\ 
\hline
$\Pro(t=1|x)$ & $q_1$ & $q_2$  & $q_3$ & $q_4$\\ 
\hline
$y^{(1)}_i${\color{orange} /uplift curve rule} & 1 {\color{orange} /+1} & 1{\color{orange} /+1}  & 0 {\color{orange} /0} & 0 {\color{orange} /0}  \\ 
\hline
$y^{(0)}_i${\color{orange} /uplift curve rule} & 0{\color{orange} /0} & 1{\color{orange} /-1}  & 0 {\color{orange} /0} & 1 {\color{orange} /-1} \\ 
\hline
$\tau_i=u$ & 1 & 0  & 0 & -1\\ 
\hline
$\uh_n(x)$ & $\alpha_1$ & \multicolumn{2}{|c|}{$\alpha_{2}=\alpha_{3}$} & $\alpha_{4}$\\ 
\hline
$\uh_d(x)$ & $\alpha_1$ & $\alpha_2$  & $\alpha_3$ & $\alpha_4$\\ 
\hline
\end{tabular}
\caption{\label{counterexample:table}Group characteristics, with the {\color{orange} uplift curve rule in orange}.
+1 means that when we sample one $x_i$ from this group, and if the latter line is sampled from the treatment variable $t_i$, then the rule makes the uplift curve going up by one for this unit $i$ added on the $x$-axis. (see uplift curve construction in Section \ref{sec:aununc-emp})}
\end{center}
\end{table}

Although the perfect Uplift model $u$ does not distinguish the 2 groups LC and ST (their ITE is the same), we consider a class of models $\mathcal{U}_d$ that have capacity to distinguish the 2 groups, as they are based on the features $x \in \{ CO, ST, LC, SD \}$. We stress that actually the perfect Uplift model $u$ belongs to $\mathcal{U}_d$, but does not have to use its capacity to distinguish ST and LC. 


We will denote hereafter by $\uh_n$ an uplift model that can\textit{n}ot distinguish ST and LC, and by $\uh_d$ an uplift model that can \textit{d}istinguish them ("\textit{n}" stands for \textit{n}ot, while "\textit{d}" stands for \textit{d}istinguishable). We compute now the slope of the uplift curve in each group. The AUUC of an uplift model depends on those slopes and how the model orders the 4 groups. 
Under $\uh_d$ and $\uh_n$, the 4 groups are \textit{a priori} not perfectly ordered. This will be determined later through the values $\alpha_i$.



\begin{table}[!ht]
\begin{center}
\begin{tabular}{ | c || c| c | c | c | } 
\hline
{\color{orange}uplift curve} & CO & ST & LC & SD \\ 
\hline \hline
$\text{slope}(u)$ & $q_1$ & \multicolumn{2}{|c|}{$q_2-1/2$} & $q_4-1$\\ 
\hline
$\text{slope}(\uh_n)$ & $q_1$ & \multicolumn{2}{|c|}{$q_2-1/2$} & $q_4-1$\\ 
\hline
$\text{slope}(\uh_d)$ & $q_1$ & $2q_2-1$ & 0 & $q_4-1$\\ 
\hline
\end{tabular}
\caption{The slopes of the empirical uplift curve.}
\label{table:V1_slope_comp}
\end{center}
\end{table}
Without loss of generality, Table \ref{table:V1_slope_comp} computes the slopes within each groups of the empirical uplift curve. These slopes do not depend on the values of $\alpha_i$'s, but only on the treatment proportions $q_i$'s.

Suppose that an uplift model $\uh$ has values $\alpha_1=0$, $\alpha_2=\alpha_3=1$ and $\alpha_4=-1$ 
with associated treatment proportions $q_1=1/4$, $q_2=5/6$, $q_3=5/12$ and $q_4=1/2$ (notice that the overall treatment ratio is 0.5). This model ranks first the ST+LC groups, then the CO group, and finally the SD group. If we compute the slopes as presented in Table \ref{table:V1_slope_comp}, one has $q_2-1/2=1/3>q_1=1/4$, which breaks the concavity of the uplift curve, associated to the exact model $\tu$. The uplift curves leading to that contradiction are displayed on Figure \ref{fig:AUUC_counter}. The AUUC of the perfect uplift model (blue uplift curve) is strictly smaller than the approximated uplift model $\uh_n$ (green curve), that is absurd because of the optimality of the perfect uplift model $\tu$.


\begin{figure}[!ht]
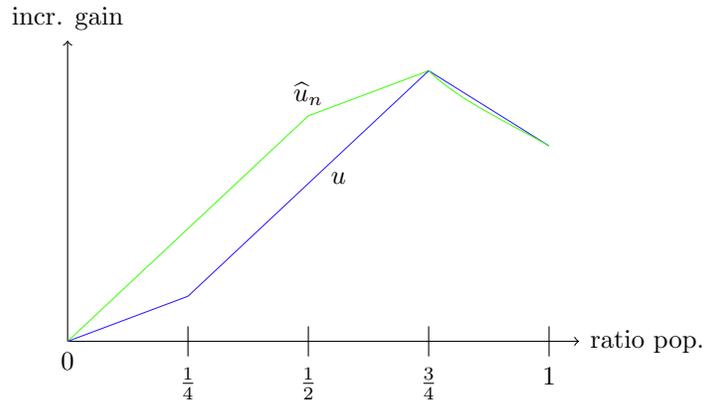

\centering
\ctikzfig{fig2}
\caption{Example of non-concavity for $\tu$. Its AUUC is not optimal. A less good uplift model, such as $\uh_n$ could have a strictly higher AUUC.}
\label{fig:AUUC_counter}
\end{figure}

\begin{center}
\begin{tcolorbox}[width=.75\linewidth, halign=center, colframe=black, colback=blue!30, boxsep=1mm, arc=3mm]
One should only use AUUC for RCT, but never with data from observational studies. More precisely, the treatment rate needs to be independent on the features $x \in \X$.
\end{tcolorbox}
\end{center}
\begin{remark}
Note that this counter-example could actually be solved by using the IPS formula as introduced in \eqref{eq:local_avg}.
\end{remark}

Now we showed that using AUUC requires to have a \textit{uniform} propensity score $$\Pro(t=1|x)=q_0, \quad \forall x \in \X,$$ one can ask whether the dataset should be balanced between treatment and control to effectively compute the AUUC. This is presented in the next subsection.

\subsection{Toy example 2: AUUC on an unbalanced dataset}
\label{sec:toy_ex2}
Suppose that we have a uniform propensity score $\Pro(t=1|x)=q_0, \quad \forall x \in \X$. Table~\ref{table:V1_slope_comp} becomes Table~\ref{table:V1_slope_comp_uniform}.

\begin{table}[!ht]
\begin{center}
\begin{tabular}{ | c || c| c | c | c | } 
\hline
{\color{orange}  uplift curve} & CO & ST & LC & SD \\ 
\hline \hline
$\text{slope}(u)$ & $q_0$ & \multicolumn{2}{|c|}{$q_0-1/2$} & $q_0-1$\\ 
\hline
$\text{slope}(\uh_n)$ & $q_0$ & \multicolumn{2}{|c|}{$q_0-1/2$} & $q_0-1$\\ 
\hline
$\text{slope}(\uh_d)$ & $q_0$ & $2q_0-1$ & 0 & $q_0-1$\\ 
\hline
\end{tabular}
\caption{The slopes of the uplift curve, if the treatment probability (a.k.a. propensity score) is uniform.}
\label{table:V1_slope_comp_uniform}
\end{center}
\end{table}

We first notice here that the uplift curve built using $u$ is concave, leading to a maximal AUUC. But let us imagine we are able to distinguish ST and LC. We call this uplift model $\uh_d$. As we have one more edge to design through $\uh_d$, we are able to build $\uh_d$ with a strictly higher AUUC. This is shown on Figure \ref{fig:AUUC_counter2}, setting $q_0=3/4$ and $\alpha_1=1$, $\alpha_2=1/2$, $\alpha_3=-1/2$ and $\alpha_4=-1$ (these values sort the 4 groups as follow: CO, ST, LC, SD).


\begin{figure}[!ht]
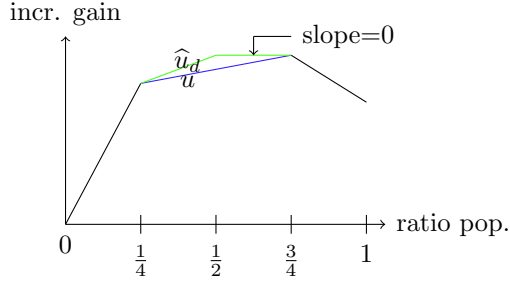

\centering
\ctikzfig{fig3}
\caption{Example when the dataset is unbalanced.}
\label{fig:AUUC_counter2}
\end{figure}

N.B.: we cannot face this issue if the dataset has been re-balanced (i.e. with $q_0=0.5$). See Section \ref{section:unbalanced_dataset} for a formal proof.

\begin{center}
\begin{tcolorbox}[width=.75\linewidth, halign=center, colframe=black, colback=blue!30, boxsep=1mm, arc=3mm]
When computing any AUUC, the dataset should be balanced.
\end{tcolorbox}
\end{center}
N.B.: We can view the re-balancing step as using an unbiased reward estimator (cf \eqref{eq:local_avg} for an IPS version).

\subsection{Toy example 3: AUUC on an unbalanced dataset, with same model capacity}
One could truly argue that the last constructed model is unfair, as the $\widehat{u}_d$ has more capacity. Indeed, model $\widehat{u}_d$ is able to distinguish ST and LC. We construct here a more general example to show that the idea is not about model capacity, but truly about unbalanced treated/untreated populations. Suppose we face 2 groups, $\X_1$ and $\X_2$ with the outcome random variables $\rew$ being Bernoulli distributed (mutually independent) as presented in Table \ref{table:V3_slope_comp_uniform}.

\begin{table}[!ht]
\begin{center}
\begin{tabular}{ | c || c| c | } 
\hline
{\color{orange} uplift curve} & $\X_1$ & $\X_2$ \\ 
\hline \hline
$\Pro(t=1\mid x \in \X_i)$ & $q_0$ & $q_0$\\ 
\hline
$\rew[1]$ & $\sim\text{Ber}(\beta_1^1)$ & $\sim\text{Ber}(\beta_2^1)$\\ 
\hline
$\rew[0]$ & $\sim\text{Ber}(\beta_1^0)$ & $\sim\text{Ber}(\beta_2^0)$\\ 
\hline
$\E[\rew[1]]$ & $\beta_1^1$ & $\beta_2^1$\\ 
\hline
$\E[\rew[0]]$ & $\beta_1^0$ & $\beta_2^0$\\ 
\hline
$\E[\tu]$ & $\beta_1^1-\beta_1^0$ & $\beta_2^1-\beta_2^0$\\ 
\hline
$\uh$ & $\alpha_1$ & $\alpha_2$\\ 
\hline
$\text{slope}$ & $q_0(\beta_1^1+\beta_1^0)-\beta_1^0$ & $q_0(\beta_2^1+\beta_2^0)-\beta_2^0$ \\ 
\hline
\end{tabular}
\caption{Generic counter example illustrating the need of using RCT 50-50\%}
\label{table:V3_slope_comp_uniform}
\end{center}
\end{table}
The slopes for $\tu$ and $\uh$ are the same within each group $\X_i$. Now the goal is to find $q_0$, $\beta_1^1$, $\beta_1^0$, $\beta_2^1$, $\beta_2^0$ such that we have an inversion of the group order. In other words, the goal is to construct a non concave uplift curve $u$, i.e. with increasing slopes (using the true uplift $\tu$), i.e. satisfying
\[ \tu(\X_1) \geq \tu(\X_2)  \text{ and } \text{slope}(\X_1)<\text{slope}(\X_2)\]
or replacing with numerical values, we find the following system of equations:
\begin{equation}
\label{eq:inversion}
\beta_1^1-\beta_1^0 \geq \beta_2^1-\beta_2^0  \text{ and } q_0(\beta_1^1+\beta_1^0)-\beta_1^0<q_0(\beta_2^1+\beta_2^0)-\beta_2^0    
\end{equation} 
These conditions are easily satisfied if $q_0 \neq .5$, by making $q_0$ very small. For numerical convenience, we exhibit here such values satisfying equation \eqref{eq:inversion}: $q_0=0.1$, $\beta_1^1=0.4$, $\beta_1^0=0.2$, $\beta_2^1=0.2$ and $\beta_2^0=0.1$.
We can then assign values to approximated uplift $\alpha_1=0.1$ and $\alpha_2=0.2$ to obtain a  model $\uh$ that outperforms the exact uplift $u$ i.e. with AUUC scores satisfying ($\text{AUUC}(\uh)>\text{AUUC}(u)$).

\begin{remark}
The latter counter-example is not possible anymore if $q_0=0.5$ (i.e. RCT 50\%-50\%), as $\text{slope}(\X_1)=0.5(\beta_1^1+\beta_1^0)-\beta_1^0=\tu(\X_1)/2$ and then the uplift curve cannot be non-concave anymore.
\end{remark}

\begin{remark}
Once again, this issue can be avoided by using the generalized formula of the uplift curve defined in Section \ref{section:unbalanced_dataset}. We can note that with this non 50\%-50\% RCT setting, we can easily estimate the propensity score $q_0$ (when not known) from data.
\end{remark}


\section{Working with Unbalanced Datasets}
\label{section:unbalanced_dataset}

The previous section presented examples showing the AUUC metric as defined in Section \ref{section:theoryauuc} should not be directly applied on an unbalanced dataset. In the current section, we show how we can manage that case and justify the $y$-axis and $x$-axis re-balancing of the uplift curve mentioned in Section \ref{sec:normtwoaxes}.

\subsection{Introducing Observed and Underlying Distributions}


Reminder: $\X \subset \R^d$ is the user feature space, with the associated probability measure $\Pro$. The full reward for every point $x \in \X$ is provided by $\rew(x) = (y^{(1)}(x), y^{(0)}(x)) \in \R^2$. The true uplift is therefore $\tau(x) = y^{(1)}(x) - y^{(0)}(x) \in \R$.

To simulate a real dataset, with a clear separation between the treatment and the control groups, we introduce $\XS$ with the common ITE estimator (\textit{signed} outcome function) $\tauh$:
\[ \XS = \X \times \{ 0,1 \} \]

\[ \tauh \colon \XS \to \R,  \quad \tauh(x, t) = y^{(1)}(x) t - y^{(0)}(x) (1 - t) = \begin{cases}
  y^{(1)}(x) & t = 1\\
  -y^{(0)}(x) & t = 0
\end{cases}
\]

$\XS$ is associated with two probability measures:
\begin{itemize}
    \item an \textit{observed} probability measure $\Qro_o$: $\Qro_o(x, t) = \Pro(x) \Sam_o(t | x)$
    \item an \textit{underlying} probability measure $\Qro_u$: $\Qro_u(x, t) = \Pro(x) \Sam_u(t)$. This is our target.
\end{itemize}

While the probability measure $\Pro$ over $\X$ is the same for both cases, the sampling measures $\Sam_o$ and $\Sam_u$ differ. The \textit{observed} sampling measure $\Sam_o$ is the cause of the possible dataset imbalance. It may or may not depend on $x$; it does not depend on $x$ in the case of a (possibly unbalanced) RCT. One constraint we need is that $\forall (x, t) \in \XS, \Pro(x) \neq 0 \implies \Sam_o(t|x) \neq 0$, i.e. for any point of $\X$ with a non zero probability, we always have access to both treatment and control parts (Assumption \ref{ass:overlap} \textit{Overlapping treatment} in Section \ref{sec:ITE=uplift}).

The \textit{underlying} probability measure is chosen as the result of a balanced RCT process: $\Sam_u(1) = \Sam_u(0) = 1/2$. By definition, it does not depend on $x$.

\subsection{$y$-axis re-balancing with $\VN$}

Using the \textit{observed} probability measure $\Qro_o$ for $\XS$, we want to compute $\VN$ and $\AUUC$ as if we observe directly the balanced \textit{underlying} probability measure $\Qro_u$ for $\XS$, or, even better, $\Pro$ for $\X$. Starting from Definition \ref{def:up_curve}, we show in Appendix \ref{appendix:unbalanced_dataset} that:

\begin{align*}
\VN(r \in [0,1]) &= \int_\X \tau(x) \one_{[\uh(x)\geq \xi(r)]} \Pro(x)\D x \quad \text{where} \quad \xi(r) =\inf \{\xi \in \R: \Pro(x : \uh(x)\geq \xi)\leq r\}\\
&= \int_{\XS} \tauh(x, t) \one_{[\uh(x)\geq \xi(r)]} 2 \Qro_u(x, t)\D (x,t) \\
&= \int_{\XS} \tauh(x, t) \one_{[\uh(x)\geq \xi(r)]} \frac{\Qro_o(x, t)}{\Sam_o(t|x)}\D (x,t) \\
\end{align*}

The $1 / \Sam_o(t|x)$ corresponds to the $y$-axis re-balancing of the uplift curve mentioned in Section  \ref{sec:defvariants}, where we presented $V(\cdot)$ functions that compute similar quantities but in different ways (absolute vs. relative incremental gains, global factors vs. local factors).

\subsection{$x$-axis re-balancing with $\AUNUCX$}

To highlight the necessary $x$-axis re-balancing when computing the $\AUNUC$, we have to get closer to the \textit{joint} setting. As explained in Section \ref{sec:normtwoaxes}, when considering a \textit{separated} setting (that corresponds to formula \eqref{auuc_def} with an integration over $[0, 1]$), the $x$-axis normalization is implicit.

Let $\VNX: \X \rightarrow \R$ be a variant function of $\VN: [0, 1] \rightarrow \R$ and $\AUNUCX: \X \rightarrow \R$ be a variant function of $\AUNUC: [0, 1] \rightarrow \R$:

\[ \VNX(x) = \int_\X \tau(\xp) \one_{[\uh(\xp) \geq \uh(x)]} \Pro(\xp)\D\xp \]

\[ \AUNUCX(x) = \int_\X \VNX(\xp) \one_{[\uh(\xp) \geq \uh(x)]} \Pro(\xp) \D(\xp) \]

In Appendix \ref{appendix:lebesgue}, we show that $\VN$ is equivalent to $\VNX$ and $\AUNUC$ is equivalent to $\AUNUCX$, i.e. that:

\[\VN(R(x)) = \VNX(x)\]
\[\AUNUC(R(x)) = \AUNUCX(x)\]

Where $R(x) \in [0, 1]$ associates $x$ to the minimal ratio $r$ such as $x$ is contained in the top-r\% when sorting the dataset by decreasing predicted uplift.

We show in Appendix \ref{appendix:unbalanced_dataset} that:
\begin{align}
\VNX(x \in \X) &= \int_\X \tau(\xp) \one_{[\uh(\xp) \geq \uh(x)]} \Pro(\xp)\D\xp \nonumber \\
&= \int_{\XS} \tauh(\xp, t) \one_{[\uh(\xp)\geq \uh(x)]} 2 \Qro_u(\xp, t)\D (\xp,t) \nonumber \\
&= \int_{\XS} \tauh(\xp, t) \one_{[\uh(\xp)\geq \uh(x)]} \frac{\Qro_o(\xp, t)}{\Sam_o(t|\xp)}\D (\xp,t) \label{form:vnx}
\end{align}

\begin{align}
\AUNUCX(x \in \X) &= \int_\X \VNX(\xp) \one_{[\uh(\xp) \geq \uh(x)]} \Pro(\xp) \D(\xp) \nonumber \\
&= \int_{\XS} \VNX(\xp) \one_{[\uh\xp)\geq \uh(x)]} \Qro_u(\xp, t)\D(\xp, t) \nonumber \\
&= \int_{\XS} \VNX(\xp) \one_{[\uh(\xp)\geq \uh(x)]} \frac{\Qro_o(\xp, t)}{2\,\Sam_o(t|\xp)}\D(\xp, t) \label{form:aunucs}
\end{align}

Now we have a $1 / \Sam_o(t|x)$ factor for both $\VNX$ ($y$-axis re-balancing of the uplift curve) and $\AUNUC$ ($x$-axis re-balancing of the uplift curve).

\subsection{Logged-bandit feedback datasets}

Reminder: the logged-bandit feedback dataset $D_N$ is a list of $N$ records:
\begin{equation*}
    D_N=\left\{ x_i, t_i, y_i^{(t_i)}, q_i \right\}_{i=1}^N, \quad \text{with} \quad q_i=\pi(t_i|x_i).
\end{equation*}

For the sake of simplicity, we consider this dataset as already sorted by decreasing predicted uplifts according to a given uplift model.

We consider here that all individuals have a distinct predicted uplift value. Managing group of individuals with a constant predicted uplift is the subject of the next section.

Formula \eqref{form:vnx} leads to definitions for $\VN$ and $V$ for the discrete case with a \textit{joint} setting: 

\[ \VN(k \in \{1...N\}) = \frac{1}{N} \sum_{i=1}^{k}{\frac{1}{q_i} \left( y^{(1)}_i t_i - y^{(0)}_i (1 - t_i) \right)} \]

\begin{align}
V(k \in \{1...N\}) = N \VN(k \in \{1...N\}) = \sum_{i=1}^{k}{\frac{1}{q_i} \left( y^{(1)}_i t_i - y^{(0)}_i (1 - t_i) \right)}  \label{form:vsum}
\end{align}

Similarly with Formula \eqref{form:aunucs} for $\AUNUC$ and $\AUUC$:

\[ \AUNUC(k \in \{1...N\}) = \frac{1}{N}\sum_{i=1}^k{\frac{\VN(i)}{2 q_i}} \]

\begin{align}
\AUUC(k \in \{1...N\}) = \frac{1}{N} \sum_{i=1}^k{\frac{V(i)}{2 q_i}}   \label{form:auucsum}
\end{align}

Note that in the RCT case, $q_i$ only depends on the sizes $|T|$ and $|C|$ of the treatment and control groups. Thus:

\[ V(k \in \{1...N\}) = \sum_{i=1}^k{\left(\frac{y^{(1)}_i t_i}{|T|} - \frac{y^{(0)}_i (1 - t_i)}{|C|} \right) (|T| + |C|)} \]

\begin{align*}
\AUUC(k \in \{1...N\}) &= \frac{1}{|T| + |C|} \sum_{i=1}^k{\left(\frac{t_i}{|T|}+ \frac{1 - t_i}{|C|} \right) V(i) \frac{|T| + |C|}{2}} \\
 &= \frac{1}{2} \sum_{i=1}^k{\left(\frac{t_i}{|T|} + \frac{1 - t_i}{|C|} \right) V(i)}
\end{align*}

When using the $\AUUC$ metric to compare several uplift models with one given dataset, we can omit the constant factors $(|T| + |C|)$ and $1 / 2$.

Finally, let's verify that the sums of the individual scaling factors along the $y$-axis (with $V$) and along the $x$-axis (with $\AUUC$) actually give the excepted values.

Along the $y$-axis, with scaling factors from $V$:
\[\sum_{i=1}^N{\left(\frac{t_i}{|T|} + \frac{(1 - t_i)}{|C|} \right) (|T| + |C|)} = \left( \sum_{i=1}^{|T|}{\frac{1}{|T|}} + \sum_{i=1}^{|C|}{\frac{1}{|C|}}\right)(|T| + |C|) = 2 (|T| + |C|) \]

Along the $x$-axis, with scaling factors from $\AUUC$:
\[ \frac{1}{2} \sum_{i=1}^N{\left(\frac{t_i}{|T|} + \frac{1 - t_i}{|C|} \right)} = \frac{1}{2} \left( \sum_{i=1}^{|T|}{ \frac{1}{|T|}} + \sum_{i=1}^{|C|}{\frac{1}{|C|}}\right)
= 1 \]

Indeed, the balanced uplift curve spans from 0 to 1 on the $x$-axis, and corresponds to a virtual dataset with twice the size of the original dataset (re-balancing effect) on the $y$-axis.

\subsection{Logged-bandit feedback datasets with iso-predicted-uplift sub-sequences}

Intuitively, when we sort the dataset by decreasing predicted uplift, the sub-sequences of individuals with the same predicted uplift are internally ordered in an indefinite way. This indeterminacy impacts the corresponding parts of the uplift curve drawn with Formula \eqref{form:vsum}, then the computed $\AUUC$ (Formula \ref{form:auucsum}).

Indeed, formulas \eqref{form:vnx} and \eqref{form:aunucs}, in particular their indicator function based on $\uh(\xp)\textcolor{red}{\geq} \uh(x)$, require that we should not split groups of users with the same predicted uplift.\\

Let $L \subseteq \{1..N\}$ be the set of the last individuals of all iso-predicted-uplift sub-sequences, i.e. $\forall i \in L: \; \uh(i) < \uh(i + 1) \lor i = N$.

One way to compute $AUUC$ with iso-predicted-uplift sub-sequences is to:
\begin{enumerate}
    \item Compute $V(k \in \{1..N\})$ (or $\VN(k)$) with Formula \eqref{form:vsum}.
    \item Only keep the $V(k \in L)$ values. Indeed, the indefinite ordering of an iso-predicted uplift sub-sequence does not impact the $V(k)$ value of its last individual.
    \item
    Re-compute $V(k \in {1..N} \setminus L)$ using a linear interpolation based on the $V(k \in L)$ values. (For all individuals sorted before the $k \in L$, use the origin $V(0) = 0$ as the left point for the interpolation.)
    \item Compute $AUUC(k \in \{1..N\})$ (or $AUNUC(k)$) using Formula \eqref{form:auucsum} with the \textit{interpolated} version of $V(k)$.
\end{enumerate}


\section{Computing the uplift curve: a new variance reduction estimator}
\label{sec:Vnu}
\subsection{$V_1$ - traditional construction}
In an RCT setting, we sort the users of the dataset from high to low predicted uplift using the permutation $\phi \in \mathfrak{S}_N$, resulting in a \textit{sorted} dataset instance $(x_{\phi(i)}, y_{\phi(i)}, t_{\phi(i)})$.

As presented in Section \ref{sec:aununc-emp}, the traditional (denoted by $V_1$ hereafter) construction of the uplift curve is as follows:
\begin{enumerate}
    \item we start from $(0,0)$ and we iterate over the sorted dataset,
    \item we go 1 up iff $t_{\phi(i)}=1$ and $y_{\phi(i)}=1$ (magenta column in Table \ref{table:v1_up}), 
    \item we go 1 down iff $t_{\phi(i)}=0$ and $y_{\phi(i)}=1$ (magenta row in Table \ref{table:v1_down}), 
    \item we stay flat otherwise.
\end{enumerate}

\begin{table}[!ht]
\begin{subtable}[h]{0.49\textwidth}
\centering
\begin{tabular}{l|l|bl}
                     &     & \multicolumn{2}{c}{$t=1$}   \\ \hline
                     &     & $y=1$         & $y=0 $        \\ \hline
 \multirow{2}{*}{$t=0$} & $y=1$ &  ST          & SD          \\ 
                     & $y=0$ & CO       & LC   
\end{tabular}
\caption{$V_1$: 1-up rule}
\label{table:v1_up}
\end{subtable}
\hfill
\begin{subtable}[h]{0.49\textwidth}
\centering
\begin{tabular}{l|l|ll}
                     &     & \multicolumn{2}{c}{$t=1$}   \\ \hline
                     &     & $y=1$         & $y=0$         \\ \hline
 \multirow{2}{*}{$t=0$} & \cellcolor{magenta} $y=1$ & \cellcolor{magenta} ST    & \cellcolor{magenta} SD        \\ 
                     & $y=0$ & CO        & LC         
\end{tabular}
\caption{$V_1$: 1-down rule}
\label{table:v1_down}
\end{subtable}
\caption{$V_1$ rules for constructing the incremental gain (groups CO, ST, SD, LC are defined in Section \ref{sec:toyexample1})}.
\end{table}

Basically, we can explain the $V_1$ rules by stating that the ST group is in both the \textit{1 up} and \textit{1 down} rules. So, \textit{on average}, the vertical contribution within the ST group cancels out.

\subsection{$V_2$ - inverted labels}

The previous rules are arbitrary. Indeed, we can design a different set of rules, called $V_2$, that is as well an unbiased point-wise estimator of the uplift curve (pointwise), as showed in Appendix \ref{appendix:v1_v2}.

As for the $V_1$ construction, users are sorted from high to low predicted uplift using the permutation $\phi \in \mathfrak{S}_N$. The new $V_2$ construction of the uplift curve is as follows
\begin{enumerate}
    \item we start from $(0,0)$ and we iterate over the sorted dataset,
    \item we go 1 up iff $t_{\phi(i)}=0$ and $y_{\phi(i)}=0$ (magenta row in Table \ref{table:v2_up}), 
    \item we go 1 down iff $t_{\phi(i)}=1$ and $y_{\phi(i)}=0$ (magenta column in Table \ref{table:v2_down}), 
    \item we stay flat otherwise.
\end{enumerate}

\begin{table}[!ht]
\begin{subtable}[T]{0.49\textwidth}
\centering
\begin{tabular}{l|l|ll}
                     &     & \multicolumn{2}{c}{$t=1$}   \\ \hline
                     &     & $y=1$         & $y=0$         \\ \hline
 \multirow{2}{*}{$t=0$} & $y=1$ &   ST          & SD          \\ 
                     & \cellcolor{magenta} $y=0$ & \cellcolor{magenta} CO & \cellcolor{magenta} LC         
\end{tabular}
\caption{$V_2$: "1-up" rule}
\label{table:v2_up}
\end{subtable}
\begin{subtable}[T]{0.49\textwidth}
\centering
\begin{tabular}{l|l|lb}
                     &     & \multicolumn{2}{c}{$t=1$}   \\ \hline
                     &     & $y=1$         & $y=0$         \\ \hline
 \multirow{2}{*}{$t=0$} & $y=1$ &  ST          & SD          \\ 
                     & $y=0$ & CO        & LC         
\end{tabular}
\caption{$V_2$: "1-down" rule}
\label{table:v2_down}
\end{subtable}
\caption{$V_2$ rules for constructing the incremental gain (groups CO, ST, SD, LC are defined in Section \ref{sec:toyexample1})}
\end{table}

Here, the LC group is in both the \textit{1 up} and \textit{1 down} rules. \textit{On average}, the vertical contribution within the LC group cancels out.

We can define $V_2$ with a similar formula as $V_1$ in Section \ref{sec:aununc-emp}:
$$\widehat{V}_2(k):=\frac{1}{N}\sum_{i=1}^k \mathbbm{1}\left(t_{\phi(i)}=0 \wedge y_{\phi(i)}=0\right)-\mathbbm{1}\left(t_{\phi(i)}=1 \wedge y_{\phi(i)}=0\right)$$
To build the uplift curve, we then join each pair of successive points $(\frac{k}{N},\widehat{V}_2(k))$, $(\frac{k+1}{N},\widehat{V}_2(k+1))$ for $k \in [N-1]$.

\subsection{Variance reduction for approximating uplift $\tu$}\label{sec:var_redic}
Without loss of generality, we assume here that the permutation is the identity, with $\phi(i)=i$, $\forall i \in [N]$.

We aim at computing the AUUC by integrating one uplift curve estimator $V_i$. The variance of the chosen estimator is crucial, as the smaller it is, the smaller the AUUC metric variance is.

We introduce here an estimator based on a linear combination of $V_1$ and $V_2$, i.e. $V_{\nu}=(1-\nu)V_1+\nu V_2$, where $\nu$ is chosen to minimize the variance of such an estimator.




In appendix \ref{appendix:v1_v2}, we show
that the variance of $V(\nu)$ is minimized when $\nu = \Pro(y=1|t=0) \Pro(t=1) + \Pro(y=1|t=1) \Pro(t=0)$. Or, equivalently, when $\nu + \Pro(y=1) = \Pro(y=1|t=0) + \Pro(y=1|t=1)$.

In particular, when the dataset is balanced (i.e. $\alpha = 0.5$), we have $\nu = \frac{p_0 + p_1}{2} = \Pro(Y=1)$.

\subsection{Numerical experiments}
The git repository can be found here
\url{https://github.com/criteo-research/uplift_modeling_metric}. 
We designed an heterogeneous synthetic dataset, where we can assess the different results of this report. More information can be found directly in the gitlab.

In figure \ref{fig:V1V2}, we display the AUUC metric distribution, while using the estimator $V_{\nu}$, for various range of the $\nu$ parameter. We can observe that the variance is minimal for some $\nu$ in between the interval $[0,1]$, as proven in Section \ref{sec:var_redic}.

On figure \ref{fig:nappe1} and \ref{fig:nappe2}, we present the variance w.r.t. various $\Pro(Y=1)$ and $\nu$. The minimal variance is reached at $\nu$ depending linearly of  $\Pro(Y=1)$, as proven in Section \ref{sec:var_redic}.

\begin{figure}[!ht]
\centering
\includegraphics[width=.7\textwidth]{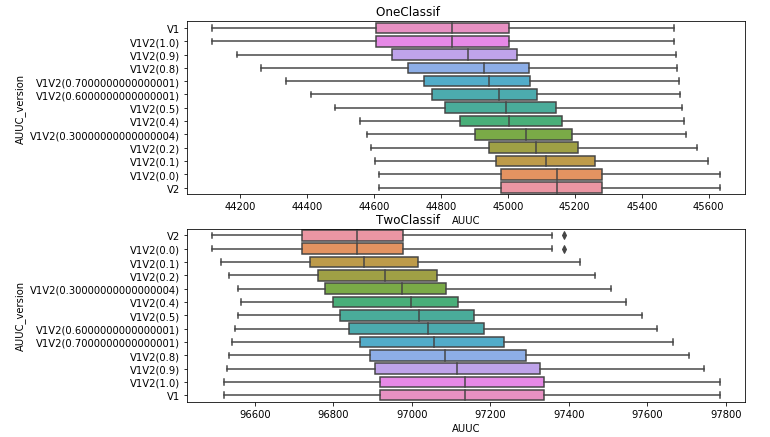}
\caption{AUUC metrics distribution $V_{\nu}$ for $\nu \in [0,1]$ computed for 101 realizations (datasets), using 2 different uplift models.}
\label{fig:V1V2}
\end{figure}

\begin{figure}[!ht]
\centering
\includegraphics[width=.7\textwidth]{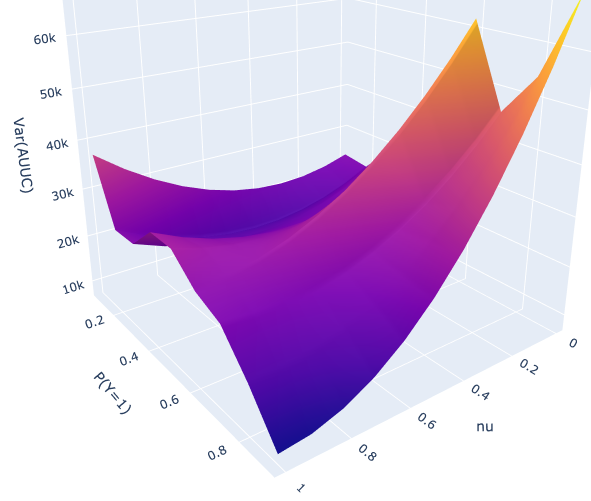}
\caption{Variance of AUUC using the $V_{\nu}$ estimator, displayed w.r.t. various $\Pro(Y=1)$ and $\nu$. Angle 1 }
\label{fig:nappe1}
\end{figure}

\begin{figure}[!ht]
\centering
\includegraphics[width=.7\textwidth]{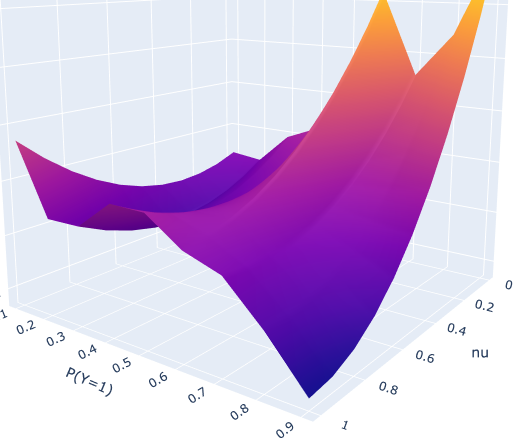}
\caption{Variance of AUUC using the $V_{\nu}$ estimator, displayed w.r.t. various $\Pro(Y=1)$ and $\nu$. Angle 1 }
\label{fig:nappe2}
\end{figure}

\section{Conclusion}
In this tech report we highlight some weaknesses of the traditional AUUC metric, denoted by $V_1$ in previous sections. More particularly, its potential very high variance, and its construction on an historically arbitrary rule. To overcome this potential problem, we proposed the 'dual' version of AUUC ($V_2$) which is build on somehow the reverted label of both treatment $T$ and outcome $Y$. Moreover, we propose a trade-off of both metrics, called $V_{\nu}$, being a linear combination of both metrics ($V_1$ and $V_2$) and we show that this linear version strictly minimizes the variance of this trade-off metric. This result is assessed in various synthetic experiments, with available code.


\bibliographystyle{abbrv}
\bibliography{./bibtex/bib}

\newpage
\appendix

\section{Rewrite $\VN$ and $\AUNUC$ formulas}
\label{appendix:lebesgue}

This appendix contains some demonstrations for section \ref{section:unbalanced_dataset}, in addition to appendix \ref{appendix:unbalanced_dataset}.

Objective: transform $\VN(r \in [0,1])$ into $\VNX(x \in \X)$ and $\AUNUC(r \in [0,1])$ into $\AUNUCX(x \in \X)$.

First, we convert Riemann integrals over $[0, 1]$ (more precisely a subset of $[0, 1]$) into Lebesgue integrals over $\X$. Then we introduce uplift-like functions. Finally we transform $\VN$ and $\AUNUC$.

\subsection{From Riemann to Lebesgue}

\subsubsection{Notations and Properties}
\label{section:riem_leb_not}

Let $\X$ be a space with associated probability measure $\Pro$, and $\Ra$ be a subset of $[0, 1]$ that is a union of closed intervals (that could be singletons), and that necessary includes $0$ and $1$.

Let $\OX$ be any function in $\Ra \to \X$, such as:

$$\forall r_1, r_2 \in \Ra, \quad r_1 \leq r_2 \implies \OX(r_1) \subseteq \OX(r_2) \quad (1)$$

$$\forall r \in \Ra \quad \Pro(\OX(r)) = r \quad (2)$$

\begin{remark*}[Proper Subsets]
From (1) and (2) we deduce that $r_1 < r_2 \implies \OX(r_1) \subset \OX(r_2)$.

\begin{proof}
$r1 \neq r2 \implies \Pro(\OX(r_1)) \neq \Pro(\OX(r_2)) \implies \OX(r_1) \neq \OX(r_2)$
\end{proof}

\end{remark*}

\begin{remark*}
$\Pro(\OX(r_1)) < \Pro(\OX(r_2)) \iff r_1 < r_2 \iff \OX(r_1) \subset \OX(r_2)$.

\begin{proof}
$\OX(r_1) = \OX(r_2) \implies \Pro(\OX(r_1)) = \Pro(\OX(r_2)) \implies r_1 = r_2$ that is a contradiction.
\end{proof}
\end{remark*}
\begin{remark*}
(2) does not necessarily mean that $\OX(0) = \emptyset$ or that $\OX(1) = \X$. But this concerns parts of $\X$ with a zero measure.
\end{remark*}
\begin{remark*}
$\OX$ is a way to define a "growing part" of the dataset. For instance $\OX(0.2)$ selects 20\% (w.r.t. the probability measure) of the dataset, while $\OX(0.3)$ selects the same 20\%, plus 10 other \%.
\end{remark*}

Let $R$ be a function in $\X \to \Ra$ such as:
$$R(x) = \min \{r \in \Ra: x \in \OX(r)\} $$

\begin{remark*}
Direct consequence: $\forall r \in [0, R(x)[ \cap \R \quad x \notin \OX(r)$
\end{remark*}

\begin{remark*}
$R(x)$ is the minimal value $r$ to catch element $x$. By definition: $x \in \OX(R(x)) \quad \forall x \in \X$.
\end{remark*}

\begin{theo}
\label{theo:alt_def_ox}
$\OX(r \in \Ra) = \{x \in \X, R(x) \leq r\}$. This will be showed in the Part 1 of the proof of Theorem \ref{theo:transfo}.
\end{theo}

$R^{-1}$ is the reverse image of $R$: $R^{-1}(A \subseteq \Ra) = \{x \in \X: R(x) \in A\}$.

In particular with $(r_1, r_2) \in \Ra^2$ and $r_1 \leq r_2$: $R^{-1}(]r_1, r_2] \cap \Ra) = \{x \in \X: r_1 < R(x) \leq r_2\}$


Let $N$ be a function in $[0, 1] \to \Ra$ such as:
$$N(r) = \min \{r^\prime \in \Ra: r^\prime \geq r\}$$

\begin{remark*}$N$ maps each open interval $\subset [0, 1] \setminus \Ra$ to the lower bound of the next closed interval of $\Ra$.
\end{remark*}

$N^{-1}$ is the reverse image of $N$: $N^{-1}(A \subseteq \Ra) = \{r \in [0, 1]: N(r) \in A\}$.

In particular with $(r_1, r_2) \in \Ra^2$ and $r_1 \leq r_2$: $N^{-1}(]r_1, r_2] \cap \Ra) = ]r_1, r_2]$ (this will be showed in the Part 2 of the proof of Theorem \ref{theo:transfo}).

Finally, let $f$ be any measurable function from $\Ra$ to $\R$ so that $f > 0$ or $E|f(R)| < \infty$.

\subsubsection{Riemann-to-Lebesgue transformation}
\begin{theo}[The Transformation Theorem]
\label{theo:transfo}

$$\int_{[0, 1]} f(N(r)) \, \D r = \int_\X f(R(x)) \, \D\Pro(x)$$

\end{theo}


\begin{proof}
We use twice \cite[Th. 1.6.9]{durrett2019probability} (\textit{Change of variables formula}).

First, by using the push-forward measure of $\Pro$ from $\X$ to $\Ra$, with a variable substitution using $R(x)$:

$$\int_\Ra f(r) \, \D (\Pro \circ R^{-1})(r) = \int_\X f(R(x)) \, \D\Pro(x)$$

Second, by using the push-forward measure of $\mu$ from $[0, 1]$ to $\Ra$, with a variable substitution using $N(r)$ (where $\mu$ is the Lebesgue measure on $\R$ as defined in \cite[Th. 1.1.4]{durrett2019probability}):

$$\int_\Ra f(r) \, \D (\mu \circ N^{-1})(r) = \int_{[0, 1]} f(N(r))\, \D r$$



So we just need to demonstrate that $\forall A \subseteq \Ra, (\Pro \circ R^{-1})(r) = (\lambda \circ N^{-1})(r)$

More precisely, we show that $\forall (r_1, r_2) \in \Ra^2$, with $r_1 \leq r_2$:

$$(\Pro \circ R^{-1})(]r_1, r_2] \cap \Ra) = (\lambda \circ N^{-1})(]r_1, r_2] \cap \Ra) = r_2 - r_1$$

\begin{proof}[Part 1: $(\Pro \circ R^{-1})$]

We first prove Theorem \ref{theo:alt_def_ox}: $\{x \in \X, R(x) \leq r\} = \OX(r) \quad \forall r \in \Ra$

($\implies$) $R(x) \leq r \implies \OX(R(x)) \subseteq \OX(r)$, by definition of $\OX$. We know that $x \in \OX(R(x))$, so $x \in \OX(r)$

($\impliedby$) Let us consider an $x \in \OX(r)$ such as $R(x) > r$. So, $\forall r^\prime < R(x) \quad x \notin \OX(r^\prime)$. In particular with $r^\prime = r$, we have a contradiction. Therefore $R(x) \leq r$.\\

$\forall (r_1, r_2) \in \Ra^2$, with $r_1 \leq r_2$:

\begin{align*}
R^{-1}(]r_1, r_2] \cap \Ra) & = \{x \in \X : r_1 < R(x) \leq r_2\} \\
&= \{x \in \X:  R(x) \leq r_2\} \setminus \{x \in \X:  R(x) \leq r_1\} \\
&= \OX(r_2) \setminus \OX(r_1) \\
\Pro(R^{-1}(]r_1, r_2] \cap \Ra)) &= \Pro(\OX(r_2)) - \Pro(\OX(r_1)) \quad \textit{$\OX(r_1)$ being a subset of $\OX(r_2)$}\\
&= r_2 - r_1
\end{align*}

\end{proof}

\begin{proof}[Part 2: $(\lambda \circ N^{-1})$]

We first show: $N^{-1}(]r_1, r_2] \cap \Ra) = ]r_1, r_2] \quad \forall (r_1, r_2) \in \Ra^2$ \textit{with}  $r_1 \leq r_2$

($\implies$) $\forall r \in ]r_1, r_2] \cap \Ra$: $N^{-1}(r) = r \in ]r_1, r_2]$

($\impliedby$) $\forall r \in ]r_1, r_2]$: $N(r) \in \Ra$ and $r \leq N(r) \leq r_2$ (by definition of $N$ and because $r_2 \in \Ra$). So $N(r) \in ]r_1, r_2] \cap \Ra$, therefore $r \in N^{-1}(]r_1, r_2] \cap \Ra)$.\\

$\forall (r_1, r_2) \in \Ra^2$, with $r_1 \leq r_2$:

\begin{align*}
    \mu(N^{-1}(]r_1, r_2] \cap \Ra)) &= \mu(]r_1, r_2]) \\
    &= r_2 - r_1 \quad \textit{\cite[Th. 1.1.4]{durrett2019probability}}
\end{align*}

\end{proof}




\end{proof}

\begin{remark*} Measure of singletons of $\Ra$: a singleton $\{r\} \subset \Ra$ has a non-zero measure iff $r > 0$ and r is a lower bound of one of the closed intervals composing $\Ra$.

\begin{proof}
($\implies$) $\forall r_2 \in \Ra$ such as $r_2 > 0$ and $(\Pro \circ R^{-1})(\{r_2\}) = a > 0$: 

$r \in ]r_2 - a, r_2[ \cap \Ra \implies \{r_2\} \subseteq ]r, r_2[ \implies R^{-1}(\{r_2\}) \subseteq R^{-1}(]r, r_2[) \implies (\Pro \circ R^{-1})(\{r_2\}) \leq (\Pro \circ R^{-1})(]r, r_2[) \implies a \leq r_2 - r \implies r \leq r_2 - a$

that is a contradiction with the definition of $r$. Thus, $]r_2 - a, r_2[ \cap \Ra = \emptyset$.


($\impliedby$)
$\forall (r_1, r_2) \in \Ra^2$, such as $r_1 < r_2$ and $]r_1, r_2[ \cap \Ra = \emptyset$:

$]r_1, r_2] = \{r_2\}$. Thus: $(\Pro \circ R^{-1})(\{r_2\}) = r_2 - r_1 > 0$ and $r_2 > 0$.

\end{proof}

\end{remark*}

\subsubsection{Application of Theorem \ref{theo:transfo} to integration on $[0, r]$}
\label{sec:app_integration}

Let $g$ be the anti-derivative function of $f \circ N$:

\begin{align*}
g(r \in [0, 1]) &= \int_0^r f(N(r^\prime)) \D r^\prime \\
&= \int_{[0, 1]} f(N(r^\prime)) \one_{[r^\prime \leq r]} \D r^\prime
\end{align*}

When $r \in \Ra$:
\begin{align*}
g(r \in \Ra) &= \int_{[0, 1]} f(N(r^\prime)) \one_{[N(r^\prime) \leq r]} \D r^\prime \quad \textit{N is increasing and} \; N(r \in \Ra) = r \\
&= \int_\X f(R(x)) \one_{[R(x) \leq r]} \D \Pro(x) \quad \textit{$\one_{[r^\prime \leq r]}$ is a measurable function}\
\end{align*}

Because $\Ra$ is a union of closed intervals that includes $0$ and $1$: when $r \notin \Ra$, $\exists! (r_1, r_2) \in \Ra^2$ such as $r \in ]r_1, r_2[$ and $]r_1, r_2[ \cap \Ra = \emptyset$. Then, $N(r) = r_2$.

\begin{align*}
g(r \in [0, 1] \setminus \Ra) &= \int_0^{r_1} f(N(r^\prime)) \D r^\prime + \int_{r_1}^r f(N(r^\prime)) \D r^\prime \\
&= g(r_1) + f(r_2) \int_{r_1}^r \D r^\prime \\
&= g(r_1) + f(r_2) (r - r_1) \\
\end{align*}

This means that $g$ fills gaps in $\Ra$ with a linear interpolation.

Hereafter, we will focus on $r$ values in $\Ra$.

\subsection{Introducing uplift-like functions}

\subsubsection{Notations and Properties}

Let $U$ be a function in $\X \to \R$ ($U$ is not necessarily injective or surjective), such as its image $\Up^*$ is a union of closed intervals and $\exists (a, b) \in \Up^{*2}, \Up^* \subseteq [a, b]$.

Let $\Up$ be a set such as $\Up = \Up^* \cup \{u\}$ with any $u > \max(\Up^*)$. $\Up$ has a minimal element (the same as $\Up^*$) and a maximal element (the chosen $u$).

Let $T$ be a function in $\Up \to [0, 1]$ such as: $T(u) = \Pro(x \in \X: U(x) \geq u)$, and let $\Ra$ be its image.

\begin{remark*}
$\{0, 1\} \subseteq \Ra$
\begin{proof}
$T(\min(\Up)) = \Pro(x \in \X: U(x) \geq \min(\Up)) = \Pro(\X) = 1$

$T(\max(\Up)) = \Pro(x \in \X: U(x) \geq \max(\Up)) = \Pro(\emptyset) = 0$ because $\forall x \in \X, U(x) \leq \max(\Up^*) < \max(\Up)$.
\end{proof}
\end{remark*}

\begin{remark*}
$T$ is decreasing.
\begin{proof}
$\forall u_1, u_2 \in \Up, u_1 \leq u_2 \iff \{x \in \X: U(x) \geq u_2\} \subseteq \{x \in \X: U(x) \geq u_1\} \iff T(u_2) \leq T(u_1)$
\end{proof}
\end{remark*}
\begin{remark*}
$T$ is not necessarily injective or surjective. 
\end{remark*}

Let $T^{-1}$ be the inverse image of $T$: $T^{-1}(\mathcal{D} \subset \Ra) = \{u \in \Up: T(u) \in \mathcal{D}\}$.

In particular, when considering singletons of $\Ra$: $T^{-1}(\{r\} \subset \Ra) = \{u \in \Up: T(u) = r\} = \{u \in \Up: \Pro(x \in \X: U(x) \geq u) = r\}$


\subsubsection{Definition of $\OX$}
Now we use $T$ and $U$ to define a $\OX$ function that respects both constraints (1) and (2) from section \ref{section:riem_leb_not}.
\begin{defi*}
$$\OX(r \in \Ra) = \{ x \in \X: T(U(x)) \leq r\}$$
\end{defi*}





$\OX$ respects both constraints (1) and (2) from section \ref{section:riem_leb_not} that are mandatory to use theorem \ref{theo:transfo}.

\begin{proof}
$$(1) \quad r_1 \leq r_2 \implies \OX(r_1) \subseteq \OX(r_2) \quad \text{by the definition of $\OX$}$$

$$(2) \quad T(U(x \in \X)) = \Pro(\xp \in \X: U(\xp) \geq U(x)) = \Pro(\xp \in \X: T(U(\xp)) \leq T(U(x))) = \Pro(\OX(T(U(x))))$$
$$\textit{therefore} \quad \forall r \in \Ra, r = \Pro(\OX(r))$$

\end{proof}

\subsubsection{Relation between $T$, $U$ and $R$}

\begin{theo}
\label{theo:tur}
$T(U(x \in \X)) = R(x)$
\end{theo}
\begin{proof}
$$\forall x \in \X, \forall r < T(U(x)), x \notin \{\xp \in \X: T(U(\xp)) \leq r \} \quad \textit{so} \quad x \notin \OX(r)$$

$$\forall x \in \X, \forall r \geq T(U(x)), x \in \{\xp \in \X: T(U(\xp)) \leq r \} \quad \textit{so} \quad x \in \OX(r)$$

Thus, $T(U(x))$ is the minimal value $r \in \Ra$ such as $x \in \OX(r)$. So $T(U(x \in \X)) = R(x)$, by definition of $R$.
\end{proof}


\subsubsection{Alternative definition of $\OX$, with $\xi$}

\begin{align*}
\OX(r \in \Ra) &= \{ x \in \X: T(U(x)) \leq r\} \\
&= \{ x \in \X: U(x) \geq u, \forall u \in T^{-1}(\{r\}) \} \quad \textit{$T$ being decreasing} \\
&= \{ x \in \X: U(x) \geq \inf(T^{-1}(\{r\})) \} \\
&= \{ x \in \X: U(x) \geq \xi_e(r) \}
\end{align*}
where $\xi_e(r \in \Ra) = \inf(T^{-1}(\{r\})) = \inf\{u \in \Up: \Pro(x \in \X: U(x) \geq u) = r\}$.

Now we consider a slightly different function $\xi$ that was defined in Definition \ref{def:up_curve} (Section \ref{section:theoryauuc}): $\xi(r \in \textcolor{red}{[0, 1]}) = \inf \{u \in \textcolor{red}{\R}: \Pro(x \in \X : U(x)\geq u)\textcolor{red}{\leq} r\}$.

Note that the domains of $r$ and $u$, and the comparison with $r$ are different.

\begin{theo}
\label{theo:xi}
$\forall r \in [0, 1],\; \xi(r) = \xi(N(r))$ with function $N$ as defined in Section \ref{section:riem_leb_not}.
\end{theo}

\begin{proof}
$ $

\begin{proof}[Part 1: from "$u \in \R$" to "$u \in \Up$"]
$ $

$\xi(r \in [0, 1]) = \inf \{u \in \textcolor{red}{\Up}: \Pro(x \in \X : U(x)\geq u)\leq r\}$

Indeed, if we consider a $r$ such as $\xi(r) \in \R \setminus \Up$ and $\up$ such as $\up = \max\{u \in \Up: u < \xi(r)\}$ : $\{x \in \X: U(x) \leq  u \} = \{x \in \X: U(x) \leq \up \}$

Hereafter, we will consider a definition of $\xi$ using $\Up$.
\end{proof}

\begin{proof}[Part 2: $\xi(r) = \xi(N(r))$]
$ $

$\forall u \in \Up, \; \Pro(x \in \X: U(x) \geq u) \in \Ra$ by definition of $T$.

Thus $\forall r \in [0, 1], \; \Pro(x \in \X : U(x)\geq u)\leq r \iff \Pro(x \in \X : U(x)\geq u)\leq \min \{r^\prime \in \Ra: r^\prime \geq r\} \iff \Pro(x \in \X : U(x)\geq u)\leq (N(r))$

\end{proof}

\end{proof}

\begin{theo}
\label{theo:xib}
$\forall r \in \Ra,\; \xi(r) = \xi_e(r)$
\end{theo}

\begin{proof}
\begin{align*}
\xi(r \in \Ra) &= \inf\{u \in \Up: \Pro(x \in \X: U(x) \geq u) \leq r\} \\
&= \inf\{u \in \Up: \Pro(x \in \X: U(x) \geq u) = \rp, \forall \rp \in \Ra, \rp \leq r\} \\
&= \inf\{\xi_e(\rp): \forall \rp \in \Ra, \rp \leq r\}
\end{align*}

\begin{align*}
\forall (\rp, r) \in \Ra^2: \quad \rp < r & \implies \OX(\rp) \subset \OX(r) \\
& \implies \{ x \in \X: U(x) \geq \xi_e(\rp) \} \subset \{ x \in \X: U(x) \geq \xi_e(r) \} \\
& \implies \xi_e(\rp) > \xi_e(r)
\end{align*}

Therefore, $\xi(r \in \Ra) = \xi_e(r)$.
\end{proof}

\begin{cor}
$\forall r \in [0, 1]\; \xi(r) = \xi(N(r)) = \xi_e(N(r))$
\end{cor}

\begin{cor}
\label{cor:ox_xr}
$\OX(r \in \Ra) = \{ x \in \X: U(x) \geq \xi(r) \}$
\end{cor}

\subsubsection{Substitution of indicator functions}

\begin{theo}[Equivalence of indicator functions]
\label{theo:subid}
$\forall (x, \xp) \in \X^2$:
\begin{align*}
\one_{[\xp \in \OX(R(x))]}
&= \one_{[R(\xp) \leq R(x)]} & \quad \textit{(a)} \\
&= \one_{[U(\xp) \geq \xi_e(R(x))]} & \quad \textit{(b)}\\
&= \one_{[U(\xp) \geq \xi(R(x))]} & \quad \textit{(c)} \\
&= \one_{[U(\xp) \geq U(x)]} & \quad \textit{(d)}\\
\end{align*}
\end{theo}
\begin{proof}

(a) is based on Theorem \ref{theo:alt_def_ox}.

(b) is based on the definition on $\OX$ using $\xi_e$, with a substitution of $r$ with $R(x)$.

(c) is based on Corollary \ref{cor:ox_xr} from (b), with a substitution of $r$ with $R(x)$.

(d):
\begin{align*}
\OX(R(x \in \X)) &= \{ \xp \in \X: T(U(\xp)) \leq R(x) \}  \quad \textit{(Definition of $\OX$)}\\
&= \{ \xp \in \X: T(U(\xp)) \leq T(U(x)) \}  \quad \textit{(Theorem \ref{theo:tur})}\\
&= \{ \xp \in \X: U(\xp) \geq U(x) \}  \quad \textit{($T$ is decreasing)}
\end{align*}
\end{proof}

In the next section, we use that theorem to substitute indicator functions within $\VN$ and $\AUNUC$ integrals.

\subsection{Transformation of $\VN$ and $\AUNUC$}

Hereafter, we call $\uh$ the $U$ function.

\subsubsection{$\VN = \VN \circ N$}

\begin{theo}
\label{theo:vn_vnn}
$\VN(r \in [0, 1]) = \VN(N(r))$
\end{theo}
\begin{proof}
\begin{align*}
\VN(r \in [0,1]) &= \int_\X \tau(x) \one_{[\uh(x)\geq \xi(r)]} \D\Pro(x) \\
&= \int_\X \tau(x) \one_{[\uh(x)\geq \xi(N(r))]} \D\Pro(x) \quad \textit{(Theorem \ref{theo:xi})}\\
&= \VN(N(r))
\end{align*}
\end{proof}

Alternative formulation for $\VN$: $\VN(r \in [0, 1]) = \int_{\OX(N(r))} \tau(x) \D\Pro(x)$

\subsubsection{From $\VN$ to $\VNX$}

Let $\VNX$ be a variant of $\VN$:
$$\VNX(x \in \X) = \int_\X \tau(\xp) \one_{[\uh(\xp) \geq \uh(x)]} \D \Pro(\xp)$$

\begin{theo}
\label{theo:vn_vnx}
$\VN(R(x \in \X)) = \VNX(x)$
\end{theo}
\begin{proof}
\begin{align*}
\VN(R(x \in \X)) &= \int_\X \tau(\xp) \one_{[\uh(\xp) \geq \xi(R(x))]} \D \Pro(\xp) \\
&= \int_\X \tau(\xp) \one_{[\uh(\xp) \geq \uh(x)]} \D \Pro(\xp) \quad \textit{(Theorem \ref{theo:subid})} \\
&= \VNX(x)
\end{align*}
\end{proof}

\subsubsection{From $\AUNUC$ to $\AUNUCX$}

Let $\AUNUCX$ be a variant of $\AUNUC$: 

$$\AUNUCX(x \in \X) = \int_\X \VNX(\xp) \one_{[\uh(\xp) \geq \uh(x)]} \D \Pro(\xp)$$

\begin{theo}
$\AUNUC(R(x \in \X)) = \AUNUCX(x)$
\end{theo}
\begin{proof}

\begin{align*}
\AUNUC(p \in [0,1]) &= \int_{[0,p]} \VN(r) \D r \\
&= \int_0^1 \VN(r) \one_{[r \leq p]} \D r \\
&= \int_0^1 \VN(N(r)) \one_{[r \leq p]} \D r \quad \textit{(Theorem \ref{theo:vn_vnn})}
\end{align*}

Now we use the results from Section \ref{sec:app_integration} to limit the domain of $\AUNUC$ to $\Ra$. Please remind that the parts of $\AUNUC$ on $[0, 1] \setminus \Ra$ (i.e. the "gaps") are filled with a linear interpolation.

\begin{align*}
\AUNUC(R(x \in \X)) &= \int_0^1 \VN(N(r)) \one_{[r \leq R(x)]} \D r \\
&= \int_0^1 \VN(N(r)) \one_{[N(r) \leq R(x)]} \D r \quad \textit{($R(x) \in \Ra$, by definition)} \\
&= \int_\X \VN(R(\xp)) \one_{[R(\xp) \leq R(x)]} \D \Pro(\xp) \quad \textit{(Theorem \ref{theo:transfo})} \\
&= \int_\X \VN(R(\xp)) \one_{[\uh(\xp) \geq \uh(x)]} \D \Pro(\xp) \quad \textit{(Theorem \ref{theo:subid}}) \\
&= \int_\X \VNX(\xp) \one_{[\uh(\xp) \geq \uh(x)]} \D \Pro(\xp) \quad \textit{(Theorem \ref{theo:vn_vnx})} \\
&= \AUNUCX(x)
\end{align*}
\end{proof}

Alternative formulations for $AUNUC$ on $\Ra$:
$$\AUNUC(r \in \Ra) = \int_{\OX(r)} \VNX(x) \D\Pro(x) = \int_{\OX(r)} \VN(R(x)) \D\Pro(x)$$

\newpage
\section{Working with Unbalanced Datasets: Demonstrations}
\label{appendix:unbalanced_dataset}

This appendix contains some demonstrations for section \ref{section:unbalanced_dataset}, in addition to appendix \ref{appendix:lebesgue}.

\subsection{$\VN$ and $\VNX$}

\begin{align*}
\VN(r \in [0,1]) &= \int_\X \tau(x) \one_{[\uh(x)\geq \xi(r)]} \Pro(x)\D x \\
&= \int_\X \left(y^{(1)}(x) - y^{(0)}(x)\right) \one_{[\uh(x)\geq \xi(r)]} \Pro(x)\D x \\
&= \int_\X \left( \sum_{t \in \{0, 1\}} \tauh(x, t) \right) \one_{[\uh(x)\geq \xi(r)]} \Pro(x)\D x \\
&= \int_\X 2 \left( \sum_{t \in \{0, 1\}} \tauh(x, t) \Sam_u(t) \right) \one_{[\uh(x)\geq \xi(r)]} \Pro(x)\D x \quad \textit{where} \quad \Sam_u(t) = 1/2 \\
&= \int_{\XS} \tauh(x, t) \one_{[\uh(x)\geq \xi(r)]} 2\,\Pro(x) \Sam_u(t)\D (x,t) \quad \textit{(Fubini's theorem)} \\
&= \int_{\XS} \tauh(x, t) \one_{[\uh(x)\geq \xi(r)]} 2\,\Qro_u(x, t)\D (x,t) \\
&= \int_{\XS} \tauh(x, t) \one_{[\uh(x)\geq \xi(r)]} 2\,\Pro(x) \Sam_u(t) \frac{\Sam_o(t|x)}{\Sam_o(t|x)}\D (x,t) \\
&= \int_{\XS} \tauh(x, t) \one_{[\uh(x)\geq \xi(r)]} \frac{\Qro_o(x, t)}{\Sam_o(t|x)}\D (x,t) \\
\end{align*}

Similarly, we show that:
\begin{align*}
\VNX(x \in \X) &= \int_\X \tau(\xp) \one_{[\uh(\xp) \geq \uh(x)]} \Pro(\xp)\D\xp \\
&= \int_{\XS} \tauh(\xp, t) \one_{[\uh(\xp)\geq \uh(x)]} 2\,\Qro_u(\xp, t)\D (\xp,t) \\
&= \int_{\XS} \tauh(\xp, t) \one_{[\uh(\xp)\geq \uh(x)]} \frac{\Qro_o(\xp, t)}{\Sam_o(t|\xp)}\D (\xp,t)
\end{align*}

\subsection{$\AUNUCX$}

\begin{align*}
\AUNUCX(x \in \X) &= \int_\X \VNX(\xp) \one_{[\uh(\xp) \geq \uh(x)]} \Pro(\xp) \D(\xp) \\
&= \int_\X \left( \sum_{t \in \{0, 1\}}{\VNX(\xp)} \Sam_u(t) \right) \one_{[\uh(\xp) \geq \uh(x)]} \Pro(\xp) \D(\xp) \quad \textit{where} \quad \Sam_u(t) = 1/2 \\
&= \int_{\XS} \VNX(\xp) \one_{[\uh\xp)\geq \uh(x)]} \Pro(\xp)\Sam_u(t)\D(\xp, t) \quad \textit{(Fubini's theorem)} \\
&= \int_{\XS} \VNX(\xp) \one_{[\uh\xp)\geq \uh(x)]} \Qro_u(\xp, t)\D(\xp, t) \\
&= \int_{\XS} \VNX(\xp) \one_{[\uh(\xp)\geq \uh(x)]} \Pro(\xp)\Sam_u(t)\frac{\Sam_o(t|\xp)}{\Sam_o(t|\xp)}\D(\xp, t) \\
&= \int_{\XS} \VNX(\xp) \one_{[\uh(\xp)\geq \uh(x)]} \frac{\Qro_o(\xp, t)}{2\,\Sam_o(t|\xp)}\D(\xp, t)
\end{align*}

\newpage
\section{Unbiasedness of $V_1$ and $V_2$}
\label{appendix:v1_v2}

In this appendix, we show that both sets of rules $V_1$ and $V_2$, as well as their linear combination $V_\nu = (1 - \nu) V_1 + \nu V_2$ are unbiased \textit{pointwise} estimators of the uplift curve.

\subsection{Distributions}

Given a dataset (i.e. a sample) of size $N$, \textit{sorted by decreasing predicted uplift}, and considering a point $r \in [0, 1]$ of the uplift curve $V$: by construction, $V(r)$ only depends on the first $rN$ individuals. More precisely, the specific ordering of those $rN$ individuals does not impact the $V(r)$ value. $V(r)$ is completely defined by the distribution of the $rN$ individuals within 8 groups: the Cartesian product of the 4 individual categories ($\{CO, ST, LC, SD\}$, defined in Section \ref{sec:toyexample1}) with the 2 treatment/control groups ($\{0, 1\}$).

For a given $p$ value, we thus consider that the each $i$ individual from the $pN$ first elements of the dataset is represented by two random variables:

\begin{itemize}
    \item i.i.d. category $C$, following a categorical distribution with 8 probabilities depending on $r$: $\beta_c(r)$ with $c$ in $\{CO, ST, LC, SD\}$.
    \item treatment $T \sim \text{Ber}(\alpha)$ where $\alpha$ is the propensity to be treated, that is independent of $C$ and $r$ (RCT assumption).
\end{itemize}

Based on definitions of the 4 categories, we can express $P(Y=1|T=1, r)$ and $P(Y=0|T=1, r)$ using $\beta_c(r)$:
\begin{align*}
p_1(r) &= P(Y=1|T=1, r) = \beta_{CO}(r) + \beta_{ST}(r) \\
p_0(r) &= P(Y=1|T=0, r) = \beta_{ST}(r) + \beta_{SD}(r)
\end{align*}

Think of $p_1(r) - p_0(r) = \beta_{CO}(r) - \beta_{SD}(r)$ as being the uplift.



\subsection{Estimators for $V_1$, $V_2$}

Let $Q_1$ be the estimator of the increment presented as $V_1$ rule, and $Q_2$ be the estimator of the increment presented as $V_2$ rule. With a balanced treatment/control distribution (i.e $\alpha = 0.5$), we can use:
\begin{itemize}
    \item $Q_1 = Y(2T-1)$
    \item $Q_2 = (Y-1)(2T-1)$
\end{itemize}

Table \ref{table:V1_V2_rules} summarizes the $V_1$ and $V_2$ sets of rules, as well as the definition of the 4 categories. Column \textit{weight} corresponds to the re-balancing factors that need to be applied on each individual contribution when dealing with unbalanced datasets. Indeed, as shown in Sections \ref{sec:toy_ex2} and \ref{section:unbalanced_dataset}, for any other values of $\alpha$, the uplift curve should not be directly built with +1 or -1 increments. Instead, we have to re-balance the two populations by applying factors on the +1/-1 increments. Those factors are based on $\alpha$ when $T=1$ and on $1 - \alpha$ when $T=0$. The re-scaled increments are called $Q_1(\alpha)$ and $Q_2(\alpha)$ in Table \ref{table:V1_V2_rules}. Note that when $alpha = 0.5$, we don't get the classical +1/-1 increments, but +2/-2 increments. We ignore the missing $1/2$ constant factor for the sake of simplicity, and it does not actually change the result of this appendix.


\begin{table}[!ht]
\begin{center}

\begin{tabular}{ | c| c || c | c | c | c | c |} 
\hline
$Y$ & $T$ & $V_1$ rules & $V_2$ rules & weight & $Q_1(\alpha)$ & $Q_2(\alpha)$\\ 
\hline 
\tikzmarknode{r1}{0} & 0 & 0 & $+1$ & $1 / (1 - \alpha)$ & 0 & $+ 1 / (1 - \alpha)$\\
\tikzmarknode{r2}{0} & 1 & 0 & $-1$ & $1 / \alpha$  & 0 & $- 1 / \alpha$\\  
\tikzmarknode{r3}{1} & 0 & $-1$ & 0 & $1 / (1 - \alpha)$ & $ - 1 / (1 - \alpha)$ & 0 \\
\tikzmarknode{r4}{1} & 1 & $+1$ & 0 & $1 / \alpha$  & $+ 1 / \alpha$ & 0 \\
\hline
\end{tabular}

\begin{tikzpicture}[<->, orange, >=stealth, overlay, remember picture,baseline={(0,0)}, bend right=90]

\node[left of=r1, node distance=2mm](rl1){};
\node[left of=r2, node distance=2mm](rl2){};
\node[left of=r3, node distance=2mm](rl3){};
\node[left of=r4, node distance=2mm](rl4){};
\node[left of=r2, node distance=2cm](rl){$CO$};
\draw [transform canvas={xshift=-3mm}] plot [smooth, tension=2] coordinates { (rl1) (rl) (rl4)};
\draw (rl1) edge node[midway,left] {$LC$} (rl2);
\draw (rl2) edge node[midway,left] {$SD$} (rl3);
\draw (rl3) edge node[midway,left] {$ST$} (rl4);
 \end{tikzpicture}
 \caption{$V_1$ and $V_2$ sets of rules, with the 4 categories and the re-scaled increments $Q_1$ and $S_2$}
\label{table:V1_V2_rules}
\end{center}
\end{table}

\subsection{Estimators for the uplift curve}

Let $\Voh(r)$, $\Vth(r)$ and $\Vnh(r)$ be point-wise estimators of the uplift curve, defined as:
\begin{align*}
\Voh(r) &= \sum_{i}^{rN}{Q_{1i}(\alpha)} \\
\Vth(r) &= \sum_{i}^{rN}{Q_{2i}(\alpha)} \\
\Vnh(r) &= (1 - \nu) \Voh(r) + \nu \Vth(r)  = \sum_{i}^{rN}{\big((1 - \nu) Q_{1i}(\alpha) + \nu Q_{2i}(\alpha)\big)} = \sum_{i}^{rN}{Q_{\nu i}(\alpha)} \quad \forall \nu \in [0, 1]
\end{align*}

All $Q_{1i}$ and $Q_{2i}$ being independent, we thus have:
\begin{align*}
E[\Voh(r)] &= rN E[Q_1(\alpha)] \\
E[\Vth(r)] &= rN E[Q_2(\alpha)] \\
E[\Vnh(r)] &= rN \big((1 - \nu)E[Q_1(\alpha)] + \nu E[Q_2(\alpha)] \big)
\end{align*}

\begin{align*}
E[Q_1(\alpha)] &= \frac{P(Y=1, T=1|r)}{\alpha} - \frac{P(Y=1, T=0|r)}{1 - \alpha} \\
&= \frac{P(Y=1|T=1, r) P(T=1)}{\alpha} - \frac{P(Y=1|T=0, r) P(T=0)}{1 - \alpha} \\
&= \frac{p_1(r) \alpha}{\alpha} -  \frac{p_0(r)(1 - \alpha)}{1 - \alpha} \\
&= p_1(r) - p_0(r)
\end{align*}

\begin{align*}
E[Q_2(\alpha)] &= \frac{P(Y=0, T=0|r)}{1 - \alpha} - \frac{P(Y=0, T=1|r)}{\alpha} \\
&= \frac{P(Y=0|T=0, r) P(T=0)}{1 - \alpha} - \frac{P(Y=0|T=1, r) P(T=1)}{\alpha} \\
&= \frac{(1 - P(Y=1|T=0, r)) P(T=0)}{1 - \alpha} - \frac{(1 - P(Y=1|T=1, r)) P(T=1)}{1 - \alpha} \\
&= \frac{(1 - p_0(r)) (1 - \alpha)}{1 - \alpha} -  \frac{(1 - p_1(r))\alpha}{\alpha} \\
&= (1 - p_0(r)) - (1 - p_1(r)) = p_1(r) - p_0(r) = E[Q_1(\alpha)]
\end{align*}

\begin{align*}
(1 - \nu)E[Q_1(\alpha)] + \nu E[Q_2(\alpha)] &= 
(1 - \nu)E[Q_1(\alpha)] + \nu E[Q_1(\alpha)] \\
&= E[Q_1(\alpha)] = p_1(r) - p_0(r)
\end{align*}

Therefore, $\forall r \in [0, 1], \forall \nu \in [0, 1]$
$$\E[\Voh(r)] = \E[\Vth(r)] = \E[\Vnh(r)] = r N (p_1(r) - p_0(r))$$

\subsection{Preparation for the variance reduction of $\Vnh$}
\begin{align*}
E[Q_1(\alpha)^2] &= \frac{P(Y=1, T=1|r)}{\alpha^2} + \frac{P(Y=1, T=0|r)}{(1 - \alpha)^2} \\
&= \frac{P(Y=1|T=1, r) P(T=1)}{\alpha^2} + \frac{P(Y=1|T=0, r) P(T=0)}{(1 - \alpha)^2} \\
&= \frac{p_1(r) \alpha}{\alpha^2} +  \frac{p_0(r)(1 - \alpha)}{(1 - \alpha)^2} \\
&= \frac{p_1(r)}{\alpha} + \frac{p_0(r)}{1 - \alpha}
\end{align*}

\begin{align*}
E[Q_2(\alpha)^2] &= \frac{P(Y=0, T=0|r)}{(1 - \alpha)^2} + \frac{P(Y=0, T=1|r)}{\alpha^2} \\
&= \frac{P(Y=0|T=0, r) P(T=0)}{(1 - \alpha)^2} + \frac{P(Y=0|T=1, r) P(T=1)}{\alpha^2} \\
&= \frac{(1 - P(Y=1|T=0, r)) P(T=0)}{(1 - \alpha)^2} + \frac{(1 - P(Y=1|T=1, r)) P(T=1)}{1 - \alpha} \\
&= \frac{(1 - p_0(r)) (1 - \alpha)}{(1 - \alpha)^2} + \frac{(1 - p_1(r))\alpha}{\alpha^2} \\
&= \frac{1 - p_0(r)}{1 - \alpha} + \frac{1 - p_1(r)}{\alpha}
\end{align*}

Thus: $$\E[Q_1(\alpha)^2] + \E[Q_2(\alpha)^2] = \frac{1}{\alpha} + \frac{1}{1 - \alpha} = \frac{1}{\alpha(1 - \alpha)}$$

\subsection{Variance reduction of $\Vnh$}

All $Q_{1i}$ and $Q_{2i}$ being independent, we have: 
$$Var(\Vnh(r)) = r N Var(Q_\nu(\alpha))$$

So, minimizing the variance of $\Vnh$ is equivalent as minimizing the variance of $Q_\nu(\alpha)$.

\begin{align*}
Var(Q_\nu(\alpha)) &= Var((1 - \nu) Q_1(\alpha) + \nu Q_2(\alpha)) \\
&= (1 - \nu)^2 Var(Q_1(\alpha)) + \nu^2 Q_2(\alpha) + 2(1 - \nu)\nu Cov(Q_1(\alpha), Q_2(\alpha))
\end{align*}

Note that $Cov(Q_1(\alpha), Q_2(\alpha))=\cancel{\E[Q_1(\alpha)  Q_2(\alpha)]}-\E[Q_1(\alpha)]\E[Q_2(\alpha)]$.\\

Differentiating $\nu \mapsto Var(Q_{\nu})$ w.r.t. $\nu$ gives:
\begin{align*}
\frac{d(Var(Q_{\nu}(\alpha)))}{d(\nu)} &= -2(1-\nu)Var(Q_1(\alpha))+2\nu Var(Q_2(\alpha))+2(1-2\nu)Cov(Q_1(\alpha),Q_2(\alpha)) \\
&= 2 \nu \big(Var(Q_1(\alpha)) + Var(Q_2(\alpha)) - 2 Cov(Q_1(\alpha),Q_2(\alpha))\big) \\
& -2 \big(Var(Q_1(\alpha)) - 2 Cov(Q_1(\alpha),Q_2(\alpha)) \big) \\
&= 2 \nu \big(\E[Q_1(\alpha)^2] - \E[Q_1(\alpha)]^2 + \E[Q_2(\alpha)^2] - \E[Q_2(\alpha)]^2 + 2 \E[Q_1(\alpha)]\E[Q_2(\alpha)] \big) \\
& - 2 \big(\E[Q_1(\alpha)^2] \; \cancel{- \E[Q_1(\alpha)]^2} + \cancel{\E[Q_1(\alpha)]\E[Q_2(\alpha)]}\big) \\
&= 2 \nu \big( \E[Q_1(\alpha)^2] + \E[Q_2(\alpha)^2] - \cancel{(\E[Q_1(\alpha)] - \E[Q_2(\alpha)])^2} \big) - 2 \E[Q_1(\alpha)^2] \\
&= \frac{2 \nu}{\alpha (1 - \alpha)} - 2 \bigg( \frac{p_1(r)}{\alpha} + \frac{p_0(r)}{1 - \alpha} \bigg)
\end{align*}

That is a linear function of $\nu$ with a positive slope: $\frac{2}{\alpha(1 - \alpha})> 0$ $\forall \alpha \in (0, 1)$.

Thus, solving $\frac{d(Var(Q_{\nu}))}{d(\nu)} = 0$ gives the $\nu$ value that minimizes $Var(Q_{\nu})$:

\begin{align*}
\nu =& \frac{\E[Q_1^2](\alpha)}{\E[Q_1^2(\alpha)] + \E[Q_2^2](\alpha)} \\
    =& \left(\frac{p_1(r)}{\alpha} + \frac{p_0(r)}{1 - \alpha}\right) \alpha (1 - \alpha) \\
    =& p_1(r)(1-\alpha) + p_0(r) \alpha
\end{align*}

Or equivalently: $$\nu + P(Y=1|r) = p_0(r) + p_1(r)$$

In particular, when $\alpha = 0.5$, we have $$\nu = \frac{p_0(r) + p_1(r)}{2} = P(Y=1|r)$$


\end{document}